\documentclass[]{article}
\usepackage[utf8]{inputenc}
\usepackage{amsmath}
\usepackage{amsfonts}
\usepackage{amssymb}
\usepackage{mathrsfs} 
\usepackage{subcaption}
\usepackage{amsthm}
\usepackage{authblk} 
\usepackage{enumerate} 
\usepackage{graphicx} 
\usepackage{algorithmic} 
\usepackage[ruled,vlined]{algorithm2e}
\usepackage{natbib} 
\usepackage{bibentry} 
\usepackage[dvipsnames]{xcolor}
\nobibliography*

\usepackage{geometry} 
\usepackage[colorlinks]{hyperref} 
\hypersetup{
	linkcolor=blue,
	citecolor=blue,
}

\newcommand{\diffd}{\mathrm{d}}
\DeclareMathOperator*{\argmin}{arg\,min}

\def\diag{\mathop{\rm diag}\nolimits}
\DeclareMathOperator{\KL}{KL}
\DeclareMathOperator{\tr}{tr}
\DeclareMathOperator{\Id}{Id}

\DeclareMathOperator{\Hneg}{H}
\DeclareMathOperator{\proj}{proj}
\DeclareMathOperator{\op}{op}
\DeclareMathOperator{\HS}{HS}
\DeclareMathOperator{\refe}{ref}
\DeclareMathOperator{\obs}{obs}

\DeclareMathOperator{\Cov}{Cov}

\numberwithin{equation}{section}
\theoremstyle{plain}
\newtheorem{thm}{Theorem}[section]
\newtheorem{rem}{Remark}[section]
\newtheorem{prop}{Proposition}[section]
\newtheorem{cor}{Corollary}[section]
\newtheorem{lem}{Lemma}[section]

\newtheorem{hyp}{Assumption}[section]

\nobibliography*
\title{Entropic optimal transport beyond product reference couplings:\\
the Gaussian case on Euclidean space}

\author[1]{Paul Freulon\footnote{paul.freulon@epfl.ch}}
\author[1]{Nikitas Georgakis\footnote{nikitas.georgakis@epfl.ch}}
\author[1]{Victor Panaretos\footnote{victor.panaretos@epfl.ch} }

\affil[1]{Ecole Polytechnique Fédérale de Lausanne, Institute of Mathematics,  CH-1015 Lausanne, Switzerland.}
\date{\today}

\begin{document}
	\maketitle
	
	\begin{abstract}
		The Optimal Transport (OT) problem with squared Euclidean cost consists in finding a coupling between two input measures that maximizes correlation. Consequently, the optimal coupling is often singular with respect to the Lebesgue measure. Regularizing the OT problem with an entropy term yields an approximation called entropic optimal transport. Entropic penalties steer the induced coupling toward a reference measure with desired properties. For instance, when seeking a diffuse coupling, the most popular reference measures are the Lebesgue measure and the product of the two input measures. In this work, we study the case where the reference coupling is not a product, focussing on the Gaussian case as a core paradigm. We establish a reduction of such a regularised OT problem to a matrix optimization problem, enabling us to provide a complete description of the solution, both in terms of the primal variable and the dual variables. Beyond its intrinsic interest, allowing non-product references is essential in dynamic statistical settings. As a key motivation, we address the reconstruction of trajectory dynamics from finitely many time marginals where, unlike product references, Gaussian process references produce transitions that assemble into a coherent continuous-time process.

	\end{abstract}
		

\tableofcontents

\section{Introduction}

\paragraph{Background.} 
Optimal transport allows to compare two probability measures by searching the most efficient way to rearrange the mass of the first measure to recover the second one. Assessing efficiency relies on a bivariate function called the ground cost function. If this ground cost function is a distance or the power of a distance, the optimal transport problem defines a distance on a subset of probability measures. This distance has found many applications in mathematics \cite{villani2009optimal, ambrosioGradientflowsmetric2005a}; for instance in the study of geometric inequalities or partial differential equations, among many other topics. In this work we have particular interest toward problems of a statistical nature \cite{panaretos2020invitation, chewi2024statistical}. In this field, the optimal transport distance allows to compare probability measures with non overlapping support. Also, in statistical problems, one might be interested in actually computing the optimal transport distance through a numerical scheme. In this case, we have to solve a linear programming problem, which can become costly in realistic settings, for example in machine learning contexts. This practical difficulty has motivated the introduction of entropic optimal transport in \cite{cuturi2013sinkhorn}. Beyond its considerable impact in computational optimal transport, this regularized version of the problem is of intrinsic interest and thus is being studied for its own sake \cite{nutz2021introduction}, for statistical interest \cite{chizat2020faster, rigollet2025sample}, and for its connections to Schr\"{o}dinger bridge problems \cite{leonardsurveySchrodinger2013}. In this work, we pursue the study of entropic optimal transport by studying one of the few cases where this problem has a closed-form: when the measures are Gaussian on the Euclidean space $\mathbb{R}^d$. In this scenario, the optimal transport problem is completely solved since \cite{givens1984class}, and then extended to a separable Hilbert space in \cite{cuesta1996lower}. The optimal transport distance when restricted to centered Gaussian measures is called the Bures-Wasserstein distance, with roots in quantum information theory \cite{nielsen2001quantum, bengtsson2017geometry}, and the geometric properties of this metric space is an active domain of research \cite{takatsu2011wasserstein, bhatia2019bures}. Consequently, the Gaussian case is of interest in its own right, beyond its role as a central test case. Accordingly, entropic optimal transport on $\mathbb{R}^d$ has been an object of study in the Gaussian case specifically, particularly when the reference measure is taken to be the product of the two marginals \cite{janati2020entropic, mallastoEntropyregularized2WassersteinDistance2022, del2020statistical} (with an extension to Hilbert spaces by \cite{minh2023,yun2025spectral}). In this paper we follow this line of work, but we study the impact of choosing a reference coupling which is not necessarily a product measure (whether of the two marginals or otherwise). To motivate why this is worthwhile, we first need to introduce some basic notions and notation related to optimal transport and its entropic version.
    
\paragraph{Entropic optimal transport} Let $\mu$ and $\nu$ be two probability measures on $\mathbb{R}^d$, and let $\Pi(\mu, \nu) := \{\pi \in \mathcal{P}( \mathbb{R}^d \times  \mathbb{R}^d)~|~ \proj_1 \# \pi = \mu ~ \text{and} \proj_2 \# \pi = \nu \}$ be the set of all measures on $\mathbb{R}^d \times \mathbb{R}^d$ that have $\mu$ and $\nu$ as first and second marginal respectively. We refer to elements of $\Pi(\mu, \nu)$ as transport plans, or couplings between $\mu$ and $\nu$. On the Euclidean space $\mathbb{R}^d$, the squared optimal transport problem between $\mu$ and $\nu$ is
\begin{equation}\label{eq:def_classic_ot}
   W_2^2(\mu, \nu) := \min_{\pi \in \Pi(\mu, \nu)} \int_{ \mathbb{R}^d \times  \mathbb{R}^d} \|x-y\|^2 d \pi(x,y).
\end{equation} 
The square root  of \eqref{eq:def_classic_ot} is a specific instance of Wasserstein distances that corresponds to the ground cost function $c(x,y)= \|x-y\|^2$. Another criterion to compare probability measures is the Kullback-Leibler divergence. This divergence, also called relative entropy, is defined for two measures $\pi$ and $\pi_{\refe}$ by the formula
\begin{equation}\label{eq:definition_kl}
	\KL(\pi|\pi_{\refe}) := \int \log \left(\frac{d\pi}{d\pi_{\refe}} \right) d\pi,
\end{equation}
where $d\pi/d\pi_{\refe}$ denotes the Radon-Nikodym derivative of $\pi$ with respect to $\pi_{\refe}$, if $\pi$ is absolutely continuous with respect to $\pi_{\refe}$. If $\pi$ is not absolutely continuous with respect to $\pi_{\refe}$, $\KL(\pi|\pi_{\refe}):= +\infty$. In entropic optimal transport, the Kullback-Leibler divergence \eqref{eq:definition_kl} is exploited as a regularizing term for the optimal transport problem \eqref{eq:def_classic_ot}. Thus, entropic optimal transport refers to problems of the form
\begin{equation}\label{eq:def_entropic_ot}
	W_{\pi_{\refe}}^{\varepsilon}(\mu, \nu): = \min_{\pi \in \Pi(\mu, \nu)}\int_{\mathbb{R}^d \times \mathbb{R}^d}\|x-y\|^2 d\pi + 2\varepsilon\KL(\pi|\pi_{\refe});
\end{equation}
where $\varepsilon \geq 0$ is a regularization parameter, and $\pi_{\refe}$ is a measure on $\mathbb{R}^d \times \mathbb{R}^d$ that we call the reference coupling.
To the best of our knowlege, there are two reference couplings have been investigated (indeed, the two are related). First, the Lebesgue measure on $\mathbb{R}^d \times \mathbb{R}^d$. In this case, the Kullback-Leibler divergence equals the negative entropy that we denote by $\Hneg$, and defined by $\Hneg(\pi) := \int\log(d\pi(x)/dx) d \pi(x)$ if $\pi$ is absolutely continuous with respect to the Lebesgue measure; and $\Hneg(\pi) = +\infty$ otherwise. Notice that, strictly speaking, the Lebesgue measure is not a coupling of $\mu$ and $\nu$, but we abuse terminology occasionally and allow ``coupling" to signify a measure on the product space. The other popular coupling is the independent coupling of the input measures. This corresponds to choosing $\pi_{\refe}=\mu \otimes \nu$, which yields the regularizing term $\KL(\pi|\mu \otimes \nu)$. These two reference couplings, that we also call reference transport plans, define the two optimization problems
\begin{equation}\label{eq:classic_eot}
    \begin{split}
    W_{\Hneg}^\varepsilon(\mu, \nu) &:= \min_{\pi \in \Pi(\mu, \nu)} \int\|x-y\|^2 d\pi(x,y) + 2\varepsilon \Hneg(\pi) \quad \text{and,} \\
     W_{\otimes}^\varepsilon(\mu, \nu) & := \min_{\pi \in \Pi(\mu, \nu)} \int\|x-y\|^2 d\pi(x,y) + 2\varepsilon \KL(\pi|\mu \otimes \nu).
    \end{split}
\end{equation}

These two regularized optimal transport problems are closely related. As pointed out for instance in \cite[Lem.~1.5]{marino2020optimal} or in  \cite[Prop.~4.2]{peyre2019computational}, we have the following correspondences. For $\pi \in \Pi(\mu, \nu)$, if $\mu$ and $\nu$ are absolutely continuous with respect to the Lebesgue measure, we have the relation
    \begin{equation*}
        \KL(\pi | \mu \otimes \nu) = \Hneg(\pi) - (\Hneg(\mu) + \Hneg(\nu)).
    \end{equation*}
    From this observation, it follows that $W_{\Hneg}^{\varepsilon}$ and $W_\otimes^\varepsilon$ relate through the equality
    \begin{equation}
        W_\otimes^\varepsilon(\mu, \nu) = W_{\Hneg}^\varepsilon(\mu, \nu) - 2\varepsilon(\Hneg(\mu) + \Hneg(\nu));
    \end{equation}
and both problems have the same solution $\pi^\varepsilon$. As stated in the abstract, the optimal transport problem with squared Euclidean cost is equivalent to searching for the coupling between $\mu$ and $\nu$ that maximizes correlation: expanding the squared Euclidean cost enables to write
\begin{equation}\label{eq:ot_maximize_correlation}
    \min_{\pi \in \Pi(\mu, \nu)} \int\|x-y\|^2 d\pi(x,y) = -  2\max_{\pi \in  \Pi(\mu, \nu)} \int \langle x,y \rangle d\pi(x,y) +  \text{constant}. 
\end{equation}
And in \eqref{eq:ot_maximize_correlation}, the right-hand side problem is a correlation-type term induced by $\pi$. So one expects the solution to be as close as possible to a deterministic linear function, subject to the marginal constraint. Indeed, Brenier's theorem \cite{brenier1991polar} states that when $\mu$ is absolutely continuous, the optimal coupling $\pi^\star$ is deterministic, in the sense that it is supported on the graph of a function $T : \mathbb{R}^d \rightarrow \mathbb{R}^d$ (said differently, $\pi^* = (\Id,T)\#\mu$).
On the other hand, the Kullback-Leibler terms in the regularized versions \eqref{eq:classic_eot} favor solutions with minimum (zero) correlation. Thus, considering $\mu \otimes \nu$ or the Lebesgue measure yields a regularization term adversarial to the optimal transport problem -- which reflects an implicit preference or prior toward diffuse couplings. 

That being said, there may well be other qualitative features that the user may want to steer the solution toward, depending on prior structural information -- in which case, one would require entropic regularisation with respect to a non-product reference. This is especially true in dynamic contexts, where the sought coupling's purpose is precisely to induce transition dynamics. A key such example is the so-called problem of \emph{trajectory inference}, where one only has access to time-marginals of a random process. Typical examples include destructive measurement regimes in biology \cite{fischer2019inferring}. 
In such a setting, the underlying dynamics cannot be observed directly, and a central question is whether one can construct meaningful dynamics from the purely static information available. Choosing a reference process, coherent dynamics can be induced by sequentially solving pairwise entropic transport problems, with reference couplings from a common continuous-time process. Here it is crucial to depart from the standard choice of product reference, which would steer toward trivial dynamics that cannot behave coherently across time-scales.\\

\color{Black}
\paragraph{Outline and contributions} In Section \ref{sec:statement_reduction}, we introduce our basic pairwise framework and show that the problem of entropic coupling with general reference is well-defined. This first part enables us to introduce our approach based on matrix analysis. The main results of this paper are in Section \ref{sec:closed_form} where we solve the pairwise Gaussian entropic optimal transport problem relative to a general reference measure. Our first proof is based on the study of the primal objective function. Then, we recover the result through the derivation and solution of the dual problem. Our main result relies on the assumption that a certain matrix is invertible. Section \ref{sec:examples} begins with a study of this assumption and provides two sufficient conditions for this to hold. In the same section, two examples of reference couplings are studied: a proper coupling only parametrized by a correlation matrix, and a reference coupling with independent coordinates. In Section \ref{sec:static_to_dyn} we bring our results to bear on the problem of trajectory reconstruction 
through a sequential implementation of our pairwise results. In that same context,  Section \ref{sec:numerical_experiments} illustrates the merits of using a general reference with way of simulating sample paths reconstructed from marginal distributions. A separate appendix collects proofs for the dual problem strategy (Section \ref{sec:proofs_dual}) and auxiliary results used in the proofs of our main statements (Section \ref{sec:auxiliary}).

\paragraph{Notation}
We use the notation $A:=B$ to say that quantity $A$ is defined by formula $B$. For $E$ and $F$ two measurable spaces, if $T : E  \rightarrow F$ is a measurable map and $\mu$ a measure on $E$, we denote by $T\#\mu$ the push-forward measure defined for every measurable set $B$ of $F$ by $T\#\mu(B) := \mu(T^{-1}(B))$. The positive integer $d$ denotes the dimension of the ambient space $\mathbb{R}^d$. The first coordinate projection $\proj_1: \mathbb{R}^d \times  \mathbb{R}^d \rightarrow \mathbb{R}^d$ is defined by $\proj_1(x,y) = x$. In the same way, for every $(x,y) \in \mathbb{R}^d \times \mathbb{R}^d$, $\proj_2(x,y) = y$. For $x,y \in \mathbb{R}^d$, $\langle x, y \rangle$ is the usual inner product defined by $\langle x, y \rangle := x^Ty$, where $x^T$ is the transpose of $x$. The space of $d\times d$ squared matrices with real coefficients is denoted by $M_d(\mathbb{R})$. For the subspace of $M_d(\mathbb{R})$ composed of symmetric matrices, we use the notation $S_d(\mathbb{R})$. The cone of positive-semidefinite matrices, i.e. matrices $M \in S_{d}(\mathbb{R})$ such that for every $x \in \mathbb{R}^d,~ \langle M x, x \rangle \geq 0$, is denoted by $S_{d}^{+}(\mathbb{R})$. In the case the symmetric matrix $M$ is such that for every $x \in \mathbb{R}^d \setminus \{0\},~ \langle Mx,x\rangle >0$, we say that $M$ positive-definite and use the notation $M \in S_{d}^{++}(\mathbb{R})$. Sometimes, if $A$ and $B$ are two symmetric matrices, we write $A \geq 0$ and $B > 0$ instead of $A \in S_d^{+}(\mathbb{R})$ and $B \in S_d^{++}(\mathbb{R})$. For $A, B$ two matrices, we denote by $\langle A, B \rangle_{\HS} := \tr(A B^T)$ the Hilbert-Schmidt inner product, also often called the Frobenius inner product. We mention that if $A, B \in S_d(\mathbb{R})$ are symmetric matrices then $\langle A, B \rangle_{\HS} = \tr(AB)$. For a positive-semidefinite matrix $A \in S_d^{+}(\mathbb{R})$, we denote by $N_d(A)$ the centered Gaussian measure defined on $\mathbb{R}^d$ with covariance matrix $A$. Similarly, if $\Sigma$ is a covariance matrix on $\mathbb{R}^d \times \mathbb{R}^d$, we use the notation $N_{2d}(\Sigma)$ for the centered Gaussian measure it defines. For a square matrix $R\in M_d(\mathbb{R})$, we denote by $\|R\|_{\op}$ the matrix norm defined by $\|R\|_{\op} := \sup_{\|x\| \leq 1} \|Rx\|$.

\section{Statement of the problem and matrix reduction}\label{sec:statement_reduction}

\subsection{Statement of the problem}
Let $\mu = N_d(A)$ and $\nu = N_d(B)$ be two centered Gaussian measures with respective covariance matrices $A$ and $B$, assumed of full rank. In this work, we investigate the case where the reference measure $\pi_{\refe}$ introduced in equation \eqref{eq:def_entropic_ot} is an arbitrary centered Gaussian measure (not necessarily having marginals $\mu$ and $\nu$; we consider that special case in Section \ref{sec:correlation_parameter}). Thus, for a user-chosen positive-definite matrix $\Sigma \in S_{2d}^{++}(\mathbb{R})$ and a regularization parameter $\varepsilon > 0$, the optimal transport problem we study is
\begin{equation}\label{eq:gaussian_kl_ot}
    W_{\Sigma}^\varepsilon(\mu, \nu) := \inf_{\pi \in \Pi(\mu, \nu)} \int_{\mathbb{R}^d \times \mathbb{R}^d} \|x-y\|^2 d\pi(x,y) + 2\varepsilon \KL(\pi|N_{2d}(\Sigma)).
\end{equation}
Hence, we are aiming to minimize the objective function 
\begin{equation}\label{eq:objective_function}
	I_{\Sigma}^{\varepsilon}(\pi): = \int_{\mathbb{R}^d \times \mathbb{R}^d} \|x-y\|^2 d\pi(x,y) + 2\varepsilon \KL(\pi|N_{2d}(\Sigma)),
\end{equation}
where $\pi$ belongs to the constraint set $\Pi(\mu,\nu)$. For Gaussian measures, problem \eqref{eq:gaussian_kl_ot} is a generalization of classic entropic optimal transport. In this problem, when the reference covariance matrix is chosen to be the block diagonal matrix $\Sigma=\diag(A, B)$, we recover the regularized optimal transport with penalty term $\KL(\cdot| \mu \otimes \nu)$. We point out that multiplying the Kullback-Leibler divergence by $2\varepsilon$ in \eqref{eq:gaussian_kl_ot} instead of 
$\varepsilon$ is arbitrary, but will remove factors $1/2$ in upcoming computations. As observed before us, for instance in \cite{marino2020optimal, leonardsurveySchrodinger2013}, the regularized problem \eqref{eq:gaussian_kl_ot} can be reformulated as a (static) Schr\"{o}dinger bridge problem. Indeed, we have the following equality 
\begin{equation}\label{eq:schroed_pb}
	W_\Sigma^{\varepsilon}(\mu, \nu) = 2\varepsilon \inf_{\pi \in \Pi(\mu, \nu)} \KL\left(\pi | e^{-\frac{\|x-y\|^2}{2 \varepsilon}}\pi_{\refe}(dxdy) \right),
\end{equation}
where in this case $\pi_{\refe} = N_{2d}(\Sigma)$. The first questions regarding \eqref{eq:gaussian_kl_ot} relate to the well-posedness of the minimization problem defining the quantity $W_\Sigma^{\varepsilon}$. To address the questions of existence and uniqueness of a solution to problem \eqref{eq:gaussian_kl_ot}, we exploit results on Schr\"{o}dinger bridge problems in the specific case of Gaussian measures. 

\begin{lem} Let $\Sigma$ be a full-rank covariance matrix acting on $\mathbb{R}^d \times \mathbb{R}^d$. Then, for any pair of Gaussian measures $\mu = N_d(A)$ and $\nu = N_d(B)$ with non-singular covariances $A$ and $B$, the regularized optimal transport problem \eqref{eq:gaussian_kl_ot} has a unique solution.
\end{lem}
\begin{proof}	
	Denote by $\mu \otimes \nu \in \Pi(\mu, \nu)$ the product measure induced by $\mu = N_d(A)$ and $\nu = N_d(B)$. 
	The transport cost of $\mu \otimes \nu$ can be explicitly computed and is given by 
	\begin{align*}
		I(\mu \otimes \nu ) & = \int_{\mathbb{R}^d}\|x \|^2 d\mu(x) + \int_{\mathbb{R}^d} \|y \|^2 d\nu(y) - 2 \int_{\mathbb{R}^d \times \mathbb{R}^d} \langle x , y \rangle d\mu\otimes \nu(x,y) \\
		& = \int_{\mathbb{R}^d}\|x \|^2 d\mu(x) + \int_{\mathbb{R}^d} \|y \|^2 d\nu(y) < + \infty.
	\end{align*}
We now study the Kulback-Leibler divergence term. As $\mu$ and $\nu$ are centered Gaussian measures, the product measure $\mu \otimes \nu$ is also centered Gaussian, and has covariance matrix $\Sigma_{\otimes}$ defined by
\begin{equation*}
	\Sigma_{\otimes} := \begin{pmatrix}
		A & 0 \\
		0 & B
	\end{pmatrix}.
\end{equation*}
From this observation and Proposition \ref{prop:kl_gaussian}, we deduce the equalities 
\begin{align*}
	2\KL(\mu\otimes \nu | N_{2d}(\Sigma)) & = 2\KL(N_{2d}(\Sigma_{\otimes})| N_{2d}(\Sigma)) \\
	& = \tr(\Sigma^{-1} \Sigma_{\otimes} - I_{2d}) - \log \det(\Sigma^{-1}\Sigma_{\otimes}).
\end{align*}
As $\Sigma $ and $\Sigma_{\otimes}$ have full rank, $\det(\Sigma^{-1}\Sigma_{\otimes}) > 0$. Hence $\KL(\mu\otimes \nu | N_{2d}(\Sigma)) < +\infty$. These computations of the transport term and the Kullback-Leibler divergence term show that the objective function \eqref{eq:objective_function} of the regularized problem \eqref{eq:gaussian_kl_ot} is finite when evaluated at $\mu \otimes \nu$. Applying \cite[Thm.~2.1]{nutz2021introduction} ensures that there exists a unique solution to regularized transport problem \eqref{eq:gaussian_kl_ot}.
\end{proof}

To derive a closed form for the optimal transport problem under study, we exploit that centered Gaussian measures are fully characterized by their covariance matrices.

\subsection{Matrix reduction}

When the input measures are centred Gaussian, the set of admissible couplings $\Pi(\mu,\nu)$ can be reduced to a set of admissible cross-covariance matrices. Specifically, for $A$ and $B$ two full-rank covariance matrices on $\mathbb{R}^d$, we introduce the convex constraint set 
\begin{equation}\label{eq:notation_admissible_cov}
	\Pi^{+}(A,B) := \left\{ C \in M_d(\mathbb{R}) ~|~ \begin{pmatrix}
		A & C \\
		C^T& B
	\end{pmatrix} \geq 0 \right\}.
\end{equation}

\begin{lem}\label{lem:linearot_gaussmat}
	Set $\mu=N_d(A)$ and $\nu = N_d(B)$ two centered Gaussian measures with covariance matrices denoted by $A$ and $B$.  The optimal transport problem between $\mu$ and $\nu$ reduces to minimizing a scalar product. Indeed, the equality
	\begin{equation}\label{eq:linearot_gaussmat}
		\min_{\pi \in \Pi(\mu, \nu)} \int_{\mathbb{R}^d \times \mathbb{R}^d} \|x-y\|^2 d \pi(x,y) = \min_{C \in \Pi^{+}(A,B)} \left \langle \begin{pmatrix}
			I_d & -I_d \\
			-I_d & I_d
		\end{pmatrix}, \begin{pmatrix}
			A & C \\
			C^T& B
		\end{pmatrix}  \right \rangle_{\rm HS}
	\end{equation}
	holds true. In last equation, the right-hand is a Hilbert-Schmidt inner product, which is defined for two arbitrary matrices $M, N \in M_d(\mathbb{R})$ by $\langle M, N \rangle_{\HS} = \tr(M N^T)$.
\end{lem}

\begin{proof}
	For $\pi \in \Pi(\mu, \nu)$, the transport cost is
	\begin{align*}
		I(\pi) & :=\int_{\mathbb{R}^d \times  \mathbb{R}^d} \|x-y\|^2 d\pi(x,y)\\
		&  = \int_{\mathbb{R}^d} \|x\|^2 d\mu(x) + \int_{\mathbb{R}^d} \|y\|^2 d\nu(y) - 2 \int_{\mathbb{R}^d \times \mathbb{R}^d} \langle x, y \rangle  d\pi(x,y)\\
		& = \tr(A) + \tr(B) - 2 \int_{\mathbb{R}^d \times \mathbb{R}^d} \langle x, y \rangle  d\pi(x,y).
	\end{align*}
	This last equation shows that the transport cost depends only on the covariance matrix of the transport plan. Moreover, for every matrix $C \in M_d(\mathbb{R})$ such that the matrix 
	\begin{equation*}
		X_C := \begin{pmatrix}
			A & C \\
			C^T & B 	
		\end{pmatrix}
	\end{equation*}
	 is positive-semidefinite, as $\mu$ and $\nu$ are Gaussian measures, the centered Gaussian measure $\pi_C := N_{2d}(X_C)$ belongs to $\Pi(\mu, \nu)$. Thus, we can parametrize the optimal transport problem as follows 
	\begin{equation}
		\min_{C \in \Pi^+(A,B)} I(\pi_C),
	\end{equation}
	where $\Pi^+(A,B)$ is the constraint set introduced in equation \eqref{eq:notation_admissible_cov}. With this new parametrization, we rewrite by $I(C) = I(\pi_C)$. That is, the objective function reads
	\begin{equation}
		I(C) =  \int_{\mathbb{R}^{2d}} \| x - y\|^2 d \pi_C(x,y).
	\end{equation}
	Doing some computations now yields 
	\begin{equation}
		I(C) = \tr(A) + \tr(B) - 2 \tr(C) = \left\langle \begin{pmatrix}
			I_d & -I_d \\
			-I_d & I_d
		\end{pmatrix}, \begin{pmatrix}
			A & C \\
			C^T & B
		\end{pmatrix}\right \rangle_{\rm HS}. 
	\end{equation}
\end{proof}

Lemma \ref{lem:linearot_gaussmat} is a classic optimal transport result when the two input measures are centered Gaussian measures. Results of this flavor can thus be found in the literature, for instance in  \cite{bhatia2019bures} or \cite[Sec.~1.6.3]{panaretos2020invitation}, as detailed in next Remark. But we explicitly recall this reduction here as it enables us to introduce our approach and notations.

\begin{rem}[Bures-Wasserstein distance]\label{rem:bures_wass}
	The problem studied in Lemma \ref{lem:linearot_gaussmat} is the $2$-optimal transport problem between Gaussian measures. This problem was already solved in \cite{givens1984class}, and extended to the case of a separable Hilbert space in \cite{cuesta1996lower}. As pointed out in the more recent work \cite{bhatia2019bures}, an alternative way to formulate the Gaussian optimal transport problem \eqref{eq:linearot_gaussmat} is 
	\begin{equation}\label{eq:bures_wass_optim_prob}
		W_2^2(\mu, \nu) = \min_{C \in M_d(\mathbb{R})} \tr(A) + \tr(B) - 2 \tr(C) \quad \text{such that} \quad \begin{pmatrix}
			A & C \\
			C^T & B
		\end{pmatrix} \geq 0.
	\end{equation}
	In the same reference, one can find the solution to problem \eqref{eq:bures_wass_optim_prob} which is given by the non-symmetric matrix $\sqrt{AB}$ defined by
	\begin{equation}\label{eq:square_root_product}
		\sqrt{AB} := A^{1/2}\left( A^{1/2}B A^{1/2} \right)^{1/2} A^{-1/2}.
	\end{equation}
	It follows that the $2$-Wasserstein distance between Gaussian measures has the closed form expression
	\begin{equation}\label{eq:solution_bures_wass}
		W_2^2(\mu, \nu) = \tr(A) + \tr(B) - 2 \tr\left( A^{1/2}B A^{1/2} \right)^{1/2}.
	\end{equation}
\end{rem}

The right-hand side of equation \eqref{eq:linearot_gaussmat} involves a Hilbert-Schmidt inner product between two matrices we will repeatedly manipulate throughout what follows. We thus introduce separate notations for these two important matrices. From Lemma \ref{lem:linearot_gaussmat}, the optimal transport problem between Gaussian measures reduces to minimizing the Hilbert-Schmidt scalar product between an admissible covariance matrix $X_C$, and another matrix $Y$ acting on $\mathbb{R}^d \times \mathbb{R}^d$. This matrix is defined by  
	\begin{equation}\label{eq:notation_otmatrix}
		Y := \begin{pmatrix}
			I_d & -I_d \\
			-I_d & I_d
		\end{pmatrix}.
	\end{equation}
	We sometimes refer to $Y$ as the optimal transport matrix. The second matrix involved in the inner product \eqref{eq:linearot_gaussmat} is a covariance matrix.  Let $\pi$ be an admissible coupling between $\mu=N_d(A)$ and $\nu = N_d(B)$. Then, there exists a squared matrix $C \in M_d(\mathbb{R})$ such that we can write the covariance matrix of $\pi$ as
	\begin{equation}\label{eq:notation_crosscov}
		X_C := \begin{pmatrix}
			A & C \\
			C^T & B
		\end{pmatrix}.
	\end{equation}
The matrix $C$ is called the cross-covariance of $\Sigma$, and if a pair of random variable $(Z_1,Z_2)$ has distribution $\pi$, then $C$ can be explicitly written as $C = \mathbb{E}[Z_1 Z_2^T]$. We will make use of notation \eqref{eq:notation_crosscov} to denote an admissible covariance matrix parametrized by its cross-covariance matrix.

\subsection{The Kullback-Leibler divergence as a regularizing term}\label{sec:gaussian_regularization}

Adding a Kullback-Leibler divergence penalty on the optimal transport problem requires some absolute continuity conditions to be satisfied. Assuming to work with full-rank covariance matrices simplifies the matter. 
In this section, we set a full-rank covariance matrix $\Sigma$ on the product space $\mathbb{R}^d \times \mathbb{R}^d$, and $\varepsilon > 0$ a parameter tuning the strength of the Kullback-Leibler divergence. With these two regularization parameters, the optimal transport problem we study is
\begin{equation}\label{eq:eot_gauss_refcpl}
  W_\Sigma^{\varepsilon}(\mu, \nu) =  \min_{\pi \in \Pi(\mu, \nu)} \int_{ \mathbb{R}^d \times  \mathbb{R}^d} \|x-y\|^2 d\pi(x,y) + 2 \varepsilon \KL(\pi|N_{2d}(\Sigma)).
\end{equation} 
The following lemma states that we can parametrize our problem through cross-covariance matrices. 

\begin{lem}\label{lem:regot_matrix}
Denote by $Y \in S_{2d}(\mathbb{R})$ the optimal transport matrix introduced in equation \eqref{eq:notation_otmatrix} and by $X_C$ an admissible covariance matrix as in equation \eqref{eq:notation_crosscov}. With these notations, the regularized optimal transport problem reduces to the minimization problem
    \begin{equation}\label{eq:regeot_matrix_problem}
        W_\Sigma^{\varepsilon}(\mu, \nu) = \min_{C \in \Pi^{+}(A,B)} \left\{\left\langle Y + \varepsilon \Sigma^{-1}, X_C
         \right\rangle_{\HS} - \varepsilon \log \det (X_C) \right\} + \varepsilon \log \det (\Sigma) - 2 \varepsilon   d.
    \end{equation}
\end{lem}

\begin{proof}
	In Lemma \ref{lem:linearot_gaussmat}, we have seen that any admissible coupling $\pi \in \Pi(\mu,\nu)$, has covariance matrix 
	\begin{equation*}
		X_C = \begin{pmatrix}
			A & C \\
			C^T & B
		\end{pmatrix}
	\end{equation*}
    with  $C \in M_d(\mathbb{R})$. In the same lemma, we have established that $N(X_C)$ is an admissible coupling has transport cost
    \begin{equation*}
        I(N(X_C)) = I(\pi) = \langle Y, X_C \rangle_{\rm HS}.
    \end{equation*}

We now turn to the penalty term. As the reference coupling in \eqref{eq:eot_gauss_refcpl} has been chosen Gaussian, we can still restrict to Gaussian coupling.
Indeed, from Lemma \ref{lem:optim_kl_gauss} we have 
\begin{equation*}
	\KL(N(X_C)|N(\Sigma)) \leq \KL(\pi|N(\Sigma)).
\end{equation*}
As $\pi$ has been chosen arbitrary we derive
\begin{equation*}
	\min_{C \in \Pi^{+}(A,B)} I(N(X_C))+2\varepsilon\KL(N(X_C)|N(\Sigma)) = \min_{\pi \in \Pi(\mu,\nu)} I(\pi) + 2\varepsilon \KL(\pi|N(\Sigma)) =: W_\Sigma^{\varepsilon}(\mu,\nu).
\end{equation*}
Next, exploiting Proposition \ref{prop:kl_gaussian}, we can rewrite 
    \begin{equation*}
        2\KL(N_{2d}(X_C), N_{2d}(\Sigma)) = \langle \Sigma^{-1}, X_C \rangle_{\rm HS} - \log \det(\Sigma^{-1} X_C) -2d.
    \end{equation*}
    Thus, the regularized transport loss can be rewritten
    \begin{equation*}
        I(N(X_C)) + 2\varepsilon \KL(N_{2d}(X_C)| N_{2d}(\Sigma)) = \langle Y+\varepsilon \Sigma^{-1}, X_C \rangle_{\rm HS} - \varepsilon\log \det(X_C) + \varepsilon\log \det(\Sigma) -2\varepsilon d.
    \end{equation*}
    Finally, this regularized optimal transport problem reads
    \begin{align*}
        W_\Sigma^\varepsilon(\mu, \nu) & = \min_{C \in \Pi^{+}(A,B)} \{ \langle Y+\varepsilon \Sigma^{-1}, X_C \rangle_{\rm HS} - \varepsilon\log \det(X_C) \} + \underbrace{\varepsilon\log \det(\Sigma) -2\varepsilon d}_{\text{independent of}~C},
    \end{align*}
    as claimed.
\end{proof}
Writing a coupling of $N(A)$ and $N(B)$ as depending of the cross-covariance $C$ only allows to remove the constraints of the problem. However, to exploit convex duality tools, it is convenient to keep in mind the constrained formulation of optimal transport. For this purpose, if $X \in S_{2d}^{+}(\mathbb{R})$ is a coupling covariance matrix, we use the block decomposition 
	\begin{equation*}
		X = \begin{pmatrix}
			X_{11} & X_{12} \\
			X_{12}^T & X_{22}
 		\end{pmatrix},
	\end{equation*}
	where all blocks have same dimension $d\times d$. With this notation, we can rewrite the matrix reduction \eqref{eq:regeot_matrix_problem} of entropic optimal transport \eqref{eq:gaussian_kl_ot} as an optimization problem with equality constraints:
	\begin{equation}\label{eq:block_constrained_primal}
		W_{\Sigma}^{\varepsilon}(\mu, \nu) = \min_{X \in S_{2d}^+(\mathbb{R})}\langle Y + \varepsilon \Sigma^{-1}, X \rangle_{\rm HS} -  \varepsilon \log \det X \quad \text{such that} \quad X_{11} = A,~ X_{22} = B, 
	\end{equation}
when forgetting about the additive constant $\varepsilon\log \det(\Sigma) -2\varepsilon d$.

\section{Closed form for arbitrary reference Gaussian measures}\label{sec:closed_form}
We may now leverage the matrix reduction of the entropic optimal transport \eqref{eq:gaussian_kl_ot} in order to deduce its solution.

\subsection{Primal problem approach}

From Lemma \ref{lem:regot_matrix}, the objective function to minimize is 
\begin{equation}\label{eq:objective_regularized_gauss}
    \begin{array}{cccc}
        I_\Sigma^\varepsilon :& \Pi^{+}(A,B) & \rightarrow & \mathbb{R} \cup \{+ \infty\} \\
         & C & \mapsto & \left\langle Y + \varepsilon \Sigma^{-1}, X_C
         \right\rangle_{\HS} - \varepsilon \log \det (X_C),
    \end{array}
\end{equation}
where $\varepsilon>0$. We begin by showing that the search space $\Pi^+(A,B)$ can be reduced.

\begin{lem}\label{lem:convexity_primal_obj}
    The objective function $I_\Sigma^\varepsilon$ introduced in equation \eqref{eq:objective_regularized_gauss} reaches its minimum on the set 
    \begin{equation}\label{eq:constraint_positive_mat}
    	\Pi^{++}(A,B) := \left\{ C \in M_d(\mathbb{R}) ~|~ X_C := \begin{pmatrix}
    		A & C \\
    		C^T& B
    	\end{pmatrix} > 0 \right\}.
    \end{equation} On this set $\Pi^{++}(A,B)$, the objective function $I_{\Sigma}^{\varepsilon}$ is strictly convex.
\end{lem}

\begin{proof}
We begin by showing that we can restrict to positive-definite covariance. For every $C \in \Pi^{+}(A,B)$ such that $X_C$ is positive-semidefinite but not positive-definite, $\det(X_C) = 0$. This implies that the objective function \eqref{eq:objective_regularized_gauss} takes value $+\infty$ when computed at such points $C$. Taking $C = 0$, as $A$ and $B$ are full rank, we have that $X_0 = \diag(A,B)$ is positive definite. This observation shows that the objective function at $X_0$ is such that $I_\Sigma^{\varepsilon}(X_0) < +\infty$. Hence, if a minimum of our objective function is reached, it is on the subset of $\Pi^+(A,B)$ of positive-definite matrices. This is precisely the set $\Pi^{++}(A,B)$ introduced in equation \eqref{eq:constraint_positive_mat}.\\
	
Regarding the strict convexity, we point out that the objective function $I_{\Sigma}^{\varepsilon}$ that maps $C$ to $\left\langle Y + \varepsilon \Sigma^{-1}, X_C
\right\rangle_{\HS} - \varepsilon \log \det (X_C)$ can be written as the composition of two functions. Specifically, $I_{\Sigma}^{\varepsilon} = \ell \circ f$, where $f$ is defined for every $C \in \Pi^{++}(A,B)$ by 
\begin{equation*}
	f(C) := X_C = \begin{pmatrix}
		A & C \\
		C^T & B
	\end{pmatrix}, 
\end{equation*}
and for $X \in S_{2d}^{++}(\mathbb{R})$, $\ell(X) := \left\langle Y + \varepsilon \Sigma^{-1}, X
\right\rangle_{\HS} - \varepsilon \log \det (X)$. Now, the function $\log \det$ is strictly concave on $S_{2d}^{++}(\mathbb{R})$ (e.g. \citep[p.~42,~Cor.~1.4.2]{bakonyi2011matrix}). From this we deduce that $\ell$ is strictly convex on $S_{2d}^{++}(\mathbb{R})$. Setting $t \in (0,1)$ and $C_0, C_1 \in \Pi^{++}(A,B)$ such that $C_0 \neq C_1$, we observe that 
\begin{align*}
	f((1-t)C_0 + t C_1) & = \begin{pmatrix}
		A & (1-t)C_0 + t C_1 \\
		(1-t)C_0^T + t C_1^T & B
	\end{pmatrix} \\
& =  (1-t) \begin{pmatrix}
	A & C_0 \\
	C_0^T & B
\end{pmatrix} + t \begin{pmatrix}
A & C_1 \\
C_1^T & B
\end{pmatrix} \\
& = (1-t)f(C_0) + tf(C_1).
\end{align*}
This last computation enables us to write 
\begin{align*}
	I_{\Sigma}^{\varepsilon}((1-t)C_0 + t C_1) & = \ell \circ f((1-t)C_0 + t C_1) \\
	& = \ell((1-t)f(C_0) + tf(C_1)) \\
	& < (1-t) \ell(f(C_0)) + t \ell(f(C_1)) \\
	& =  (1-t) 	I_{\Sigma}^{\varepsilon}(C_0) + t 	I_{\Sigma}^{\varepsilon}(C_1),
\end{align*}
where the inequality comes from the strict convexity of $\ell$. This shows the strict convexity of the objective function on $\Pi^{++}(A,B)$.
\end{proof}

From Lemma \ref{lem:convexity_primal_obj}, we will be able to detect the solution of our problem through the study of its critical point. We study the first variation of objective function \eqref{eq:objective_regularized_gauss}. For this purpose, we introduce the matrix $\Gamma \in S_{2d}^{++}(\mathbb{R})$ to denote the inverse of the reference covariance matrix $\Sigma$. Thus, from now on
\begin{equation}\label{eq:inverse_prior}
	\Gamma := \Sigma^{-1}, \quad \text{and we use the block formulation} \quad \Gamma =  \begin{pmatrix}
		\Gamma_{11} & \Gamma_{12} \\
		\Gamma_{12}^T& \Gamma_{22}
	\end{pmatrix},
\end{equation}
where the blocks $\Gamma_{11}, \Gamma_{12}$ and $\Gamma_{22}$ are all $d \times d$ matrices. In what comes next, we will have a particular interest for the off-diagonal block $\Gamma_{12}$ and more precisely for the matrix $I_d - \varepsilon \Gamma_{12}$. As this matrix will appear often throughout, we will denote it by 
\begin{equation}\label{eq:m_eps_matrix}
	M_\varepsilon := I_d - \varepsilon \Gamma_{12}.
\end{equation}

In classic entropic optimal transport, that is when the reference measure is the product of the input measures $\Sigma = \diag(A, B)$. This implies that $\Gamma_{12} = 0$, so that $M_\varepsilon$ reduces to the identity matrix $I_d$. In our work, not having access to this reduction adds an extra layer of technicalities.

\begin{prop}\label{prop:gradient_primal}
    The objective function \eqref{eq:objective_regularized_gauss} is differentiable at every $C \in M_d(\mathbb{R})$ such that $X_C$ is positive-definite. Furthermore, the gradient at $C \in M_d(\mathbb{R})$ such that $X_C$ is positive definite is given by the formula
    \begin{equation}\label{eq:gradient_primal}
        \nabla I_{\Sigma}^\varepsilon(C) =  2(\varepsilon A^{-1} C (B- C^T A^{-1}C)^{-1}- M_\varepsilon),
    \end{equation}
    where $M_{\varepsilon}=I_d - \varepsilon \Gamma_{12}$ as per definition \eqref{eq:m_eps_matrix}.
\end{prop}

\begin{proof}
The objective function $I_\Sigma^\varepsilon$ is the sum of a linear term denoted by $L$ and the $\log\det$ function. We first compute the gradient of the linear term defined by $L(C) := \langle Y + \varepsilon \Sigma^{-1}, X_C \rangle$. Set $C \in \Pi^{++}(A,B)$. For $H  \in M_d(\mathbb{R})$ sufficiently small so that $X_{C+H}$ is positive-definite, and using the notation $M_\varepsilon = I_d - \varepsilon \Gamma_{12}$ we can write
    \begin{align*}
        L(C+H) &=  \left\langle Y + \varepsilon \Sigma^{-1}, \begin{pmatrix}
        A & C + H \\
        C^T+ H^T& B
    \end{pmatrix}\right \rangle_{\HS}  \\
    & = \left\langle  \begin{pmatrix}
    	I_d & -I_d \\
    	-I_d & I_d
    \end{pmatrix} + \varepsilon  \begin{pmatrix}
    \Gamma_{11} & \Gamma_{12} \\
    \Gamma_{12}^T& \Gamma_{22}
\end{pmatrix}, \begin{pmatrix}
        A & C \\
        C^T& B
    \end{pmatrix} + \begin{pmatrix}
        0 &  H \\
         H^T& 0
    \end{pmatrix}\right \rangle_{\HS}  \\
	    & = \left\langle  \begin{pmatrix}
		I_d + \varepsilon \Gamma_{11} & -M_\varepsilon \\
		-M_\varepsilon^T& I_d + \varepsilon \Gamma_{22}
	\end{pmatrix}, \begin{pmatrix}
		A & C \\
		C^T& B
	\end{pmatrix} + \begin{pmatrix}
		0 &  H \\
		H^T& 0
	\end{pmatrix}\right \rangle _{\HS}  \\
    & = L(C) -\langle  2 M_{\varepsilon}, H \rangle_{\HS} .
    \end{align*}
From this computation we deduce that the linear part of objective function has gradient ${\nabla L(C) = -2M_\varepsilon}$. We now compute the gradient of the $\log\det$ function at $C$. 
\begin{align*}
    \log \det(X_{C+H}) & = \log \det \left(\begin{pmatrix}
        A & C \\
        C^T& B
    \end{pmatrix} + \begin{pmatrix}
        0 &  H \\
         H^T& 0
    \end{pmatrix} \right) \\
    & = \log \det \begin{pmatrix}
        A & C \\
        C^T& B
    \end{pmatrix} + \left\langle \begin{pmatrix}
        A & C \\
        C^T& B
    \end{pmatrix}^{-1},  \begin{pmatrix}
        0 &  H \\
         H^T& 0
    \end{pmatrix} \right\rangle_{\HS}  + o(H).
\end{align*}
To derive the last equality, we have used that the gradient of the $\log \det $ function at a matrix $X \in S_{2d}^{++}(\mathbb{R})$ is $X^{-1}$ (see e.g. \cite[p.~641]{boyd2004convex}). We now exploit the formula for computing the inverse of a block matrix. For this purpose we introduce the Schur complement defined by $S := B- C^TA^{-1}C$ and write 
\begin{align*}
    \left\langle \begin{pmatrix}
        A & C \\
        C^T& B
    \end{pmatrix}^{-1},  \begin{pmatrix}
        0 &  H \\
         H^T& 0
    \end{pmatrix} \right\rangle_{\HS} & =   \left\langle \begin{pmatrix}
        (-) & -A^{-1}CS^{-1} \\
        -S^{-1}C^TA^{-1} & (-)
    \end{pmatrix},  \begin{pmatrix}
        0 &  H \\
         H^T& 0
    \end{pmatrix} \right\rangle_{\HS} \\
    & = \langle -2A^{-1}C S^{-1}, H \rangle_{\HS}.
\end{align*}
This last computation shows that $\nabla \log \det(X_C) = -2A^{-1}C S^{-1}$.
Note that we did not need to compute the  off-diagonal blocks of the matrix $X_C^{-1}$ for the last computation. Collecting the pieces, we deduce that the gradient of our objective function is given by
\begin{equation}
    \nabla I_\Sigma^\varepsilon(C) = 2(\varepsilon A^{-1}C ( B- C^TA^{-1}C)^{-1}-M_\varepsilon).
\end{equation}
\end{proof}

We have computed the gradient of the objective function $I_\Sigma^{\varepsilon}$ in Proposition \ref{prop:gradient_primal}, and now aim to solve the equation $\nabla I_\Sigma^\varepsilon(C) =0$. To solve this equation, we need the matrix $ M_\varepsilon :=  I_d - \varepsilon \Gamma_{12}$ to be invertible. As of now, we take for granted that this assumption is true.
\begin{hyp}[Invertibility of $M_\varepsilon$]\label{hyp:invertible_mat}
	The matrix $\Sigma \in S_{2d}^{++}(\mathbb{R})$ and the parameter $\varepsilon > 0$ are chosen such that the matrix $M_\varepsilon$ introduced in equation \eqref{eq:m_eps_matrix} is invertible.
\end{hyp}
In Section \ref{sec:inv_meps} we return to the study of Assumption \ref{hyp:invertible_mat} and show that invertibility \emph{generically} holds true. More precisely, in Lemma \ref{lem:meps_invertible_surely}, we will establish  $M_{\varepsilon}$ is invertible for almost all $\varepsilon$ (i.e. except on a set of probability zero).
We now give our main result: the explicit solution to the entropic Gaussian optimal transport problem when the reference coupling is a Gaussian measure on the product space $\mathbb{R}^d \times \mathbb{R}^d$.

\begin{thm}\label{thm:closed_form_gauss}
    Let $\mu=N_d(A)$ and $\nu = N_d(B)$ be two centered Gaussian measures with non-singular covariance matrices $A$ and $B$. Assume that the reference coupling is the Gaussian measure $N_{2d}(\Sigma)$ on $\mathbb{R}^{2d}$, where $\Sigma$ is a full-rank covariance matrix having inverse $\Gamma := \Sigma^{-1} \in S_{2d}^{++}(\mathbb{R})$. Then, for all $\varepsilon > 0$ so that Assumption \eqref{hyp:invertible_mat} holds, the regularized optimal transport problem \eqref{eq:gaussian_kl_ot} has a unique solution given by the Gaussian measure
    \begin{equation}
    \pi_\varepsilon^{\star} := N_{2d}\begin{pmatrix}
        A & C_\varepsilon\\
        C_\varepsilon^T& B
    \end{pmatrix},
    \end{equation}
    where the cross-covariance matrix $C_\varepsilon$ is given by the formula
    \begin{equation}
        C_\varepsilon =\left[\left( AM_\varepsilon BM_\varepsilon^T +\frac{\varepsilon^2}{4}I_d\right)^{1/2} - \frac{\varepsilon}{2}I_d \right](M_\varepsilon^T)^{-1},
    \end{equation}
with $M_\varepsilon=I_d - \varepsilon \Gamma_{12}$ defined in equation \eqref{eq:m_eps_matrix}.
\end{thm}

\begin{rem}
	As the matrix $AM_\varepsilon B M_{\varepsilon}^T$ is not necessarily symmetric, the matrix $( AM_\varepsilon BM_\varepsilon^T +\varepsilon^2/4I_d)^{1/2}$ is defined by the formula
	\begin{equation}\label{eq:square_root_mat_regularized}
		\left(AM_\varepsilon BM_\varepsilon^T +\frac{\varepsilon^2}{4}I_d\right)^{1/2} := A^{1/2}\left( A^{1/2}M_{\varepsilon}BM_{\varepsilon}^T A^{1/2} +\frac{\varepsilon^2}{4}I_d\right)^{1/2}A^{-1/2}.
	\end{equation}
	Note that it matches the square root matrix of $AB$ introduced for the case $\varepsilon=0$ in equation \eqref{eq:square_root_product}.
\end{rem}

Before proving Theorem \ref{thm:closed_form_gauss}, we derive as a corollary the value of the entropic optimal transport cost $W_\Sigma^{\varepsilon}(\mu, \nu)$.
\begin{cor}\label{cor:closed_form_gauss}
If $\mu=N_d(A)$ and $\nu = N_d(B)$ are two centered Gaussian measures with non singular covariance matrices, the entropic optimal transport cost has the closed form expression
    \begin{align}
    W_\Sigma^{\varepsilon}(\mu, \nu)& = \tr(A) + \tr(B) - 2 \tr\left( \left(AM_\varepsilon B M_\varepsilon^T+ \frac{\varepsilon^2}{4}I_d \right)^{1/2} \right) + \varepsilon \log \det\left( \left(A M_\varepsilon B M_\varepsilon^T+ \frac{\varepsilon^2}{4}I_d\right)^{1/2}+\frac{\varepsilon}{2}I_d  \right)  \nonumber \\
   &+\varepsilon \tr (\Gamma_{11} A )  +\varepsilon \tr (\Gamma_{22} B) - \varepsilon \log \det(AB) - \varepsilon d - \varepsilon d \log(\varepsilon) - \varepsilon\log \det(\Sigma^{-1}).
    \end{align}
\end{cor}

Reassuringly, in the specific case where $\Sigma$ is chosen to be a block diagonal matrix, we recover known results.

\begin{rem}[Product measure as reference] In classic entropic optimal transport, the reference measure is $\mu \otimes \nu$, which has covariance $\diag(A,B)$. In this case, we recover the solution to classic entropic optimal transport between Gaussian measures. Indeed, when $\Gamma_{12} = 0$  the cross-covariance $C_\varepsilon$ defined in Theorem \ref{thm:closed_form_gauss} and solution of the entropic problem is 
	\begin{equation}
		C_\varepsilon = \left[\left(AB + \frac{\varepsilon^2}{4}I_d\right)^{1/2}- \frac{\varepsilon}{2}I_d\right].
	\end{equation}
And as the reference matrix  is $\diag(A,B)$, it follows that $\Gamma_{11} = A^{-1}$ and $\Gamma_{22} = B^{-1}$. In such a case, the entropic optimal transport cost given in Corollary \ref{cor:closed_form_gauss} reads
\begin{align}
	W_\otimes^{\varepsilon}(\mu, \nu)& = \tr(A) + \tr(B) - 2 \tr\left( \left(A  B + \frac{\varepsilon^2}{4}I_d \right)^{1/2} \right) + \varepsilon \log \det\left( \left(A B + \frac{\varepsilon^2}{4}I_d\right)^{1/2}+\frac{\varepsilon}{2}I_d  \right)  \nonumber \\
	&+\varepsilon d (1 -  \log(\varepsilon) ).
\end{align} 
Thus, we recover the results established in \cite{janati2020entropic} and in \cite{mallastoEntropyregularized2WassersteinDistance2022}.
\end{rem}

We now prove Theorem \ref{thm:closed_form_gauss}.

\begin{proof}
We aim to solve the gradient equation $\nabla I_{\Sigma}^\varepsilon(C)= 0$, where the gradient of the objective function is computed in Proposition \ref{prop:gradient_primal}. This equation is equivalent to
    \begin{align*}
        \varepsilon A^{-1}C S^{-1}- M_{\varepsilon}= 0 \quad & \Leftrightarrow \varepsilon A^{-1}C S^{-1} = M_{\varepsilon}\\
       & \Leftrightarrow \varepsilon C = A M_{\varepsilon} S  \\
       & \Leftrightarrow \varepsilon C = A M_{\varepsilon}(B - C^TA^{-1}C)  \\
        & \Leftrightarrow  A M_{\varepsilon} C^TA^{-1}C + \varepsilon C - AM_{\varepsilon}B=0 \\
       &\Leftrightarrow  M_{\varepsilon} C^T(A^{-1}CM_{\varepsilon}^T) + \varepsilon A^{-1}CM_{\varepsilon}^T- M_{\varepsilon}BM_{\varepsilon}^T=0 \\
       & \Leftrightarrow  (A^{-1}C M_{\varepsilon}^T)^TA (A^{-1}C M_{\varepsilon}) +\varepsilon A^{-1}CM_{\varepsilon}^T- M_{\varepsilon}B M_{\varepsilon}^T=0.
    \end{align*}
    Introducing the notation $Z=A^{-1}CM_{\varepsilon}^T$, we can rewrite the last matrix equation
    \begin{equation}
        Z^TAZ + \varepsilon Z - M_{\varepsilon}BM_{\varepsilon}^T=0.
    \end{equation}
    Taking the transpose of this new equation, and exploiting that $A$ and $B$ are symmetric matrices, we derive
    \begin{equation}
        Z^TAZ + \varepsilon Z^T- M_{\varepsilon}BM_{\varepsilon}^T=0.
    \end{equation}
    Combining these last two equations implies that $Z = Z^T$. As $Z = A^{-1}CM_{\varepsilon}^T$, a matrix $C$ solution of the equation $\nabla I_{\Sigma}^\varepsilon(C)=0$ is such that $M_{\varepsilon}C^TA^{-1}=A^{-1}C M_{\varepsilon}^T$. Using this relation, we can rewrite the equation $\nabla I_{\Sigma}^\varepsilon(C)=0$ as 
    \begin{align*}
        M_{\varepsilon}C^T(A^{-1}C M_{\varepsilon}^T) +\varepsilon A^{-1}CM_{\varepsilon}^T- M_{\varepsilon}BM_{\varepsilon}^T=0 & \Leftrightarrow (M_{\varepsilon}C^TA^{-1})CM_{\varepsilon}^T+\varepsilon A^{-1}CM_{\varepsilon}^T- M_{\varepsilon}BM_{\varepsilon}^T=0 \\
        & \Leftrightarrow A^{-1}CM_\varepsilon^T CM_{\varepsilon}^T+\varepsilon A^{-1}CM_{\varepsilon}^T- M_{\varepsilon}BM_\varepsilon^T=0 \\
        & \Leftrightarrow CM_{\varepsilon}^TCM_{\varepsilon}^T+ \varepsilon CM_{\varepsilon}^T- AM_{\varepsilon}BM_{\varepsilon}^T= 0.
    \end{align*}
After introducing the matrix $W = CM_{\varepsilon}^T$, we reach the equation
\begin{equation}\label{eq:mat_eq_primal}
 W^2 + \varepsilon W - AM_{\varepsilon}BM_{\varepsilon}^T= 0.
\end{equation}
A similar matrix equation was studied in \cite[Prop.~3]{janati2020entropic}. We adapt their analysis to solve \eqref{eq:mat_eq_primal}. First, we rewrite $CM_{\varepsilon}^T= A(A^{-1}CM_{\varepsilon}^T)$. We have noticed that if $C$ is solution of the equation $\nabla I^{\varepsilon}(C) = 0$, then $A^{-1}CM_{\varepsilon}^T$ is symmetric. Exploiting this observation, we rewrite $CM_{\varepsilon}^T$ as
\begin{align*}
    CM_{\varepsilon}^T& = A (A^{-1}CM_{\varepsilon}^T)\\
    & = A^{1/2}( A^{1/2} (A^{-1}CM_{\varepsilon}^T) A^{1/2})A^{-1/2}.
\end{align*}
As $A^{1/2} (A^{-1}CM_{\varepsilon}^T) A^{1/2}$ is symmetric, there exists $U$ orthogonal and $D$ diagonal such that 
\begin{equation*}
A^{1/2} (A^{-1}CM_{\varepsilon}^T) A^{1/2} = U^T D U.
\end{equation*}
Introducing the change of basis matrix $P = UA^{-1/2}$ we can finally write 
\begin{equation*}
    CM_{\varepsilon}^T= P^{-1}D P.
\end{equation*}
Plugging this expression in equation \eqref{eq:mat_eq_primal}, and introducing $R$ the matrix corresponding to $AM_{\varepsilon}B M_{\varepsilon}^T$ after change of bases with the matrix $P$, yields
\begin{equation}\label{eq:matrix_equation_diagonal}
    P^{-1}D^2 P + \varepsilon P^{-1}D P - P^{-1}RP = 0.
\end{equation}
This last equation implies that the matrix $AM_{\varepsilon}B M_{\varepsilon}^T$ is diagonal in the same basis as $CM_\varepsilon$. Denoting by $r_i$ the diagonal coefficients of the matrix $R$, solving last matrix equation boils down to solving the $d$ quadractic real equations
\begin{equation}
    \delta_i^2 + \varepsilon \delta_i - r_i = 0,
\end{equation}
with respect to $\delta_i$. This equation has two solutions 
\begin{equation*}
    \delta_i^{-} = -\frac{\varepsilon}{2} - \sqrt{\frac{\varepsilon^2}{4} + r_i} \quad \text{and} \quad \delta_i^+ = -\frac{\varepsilon}{2} + \sqrt{\frac{\varepsilon^2}{4} + r_i}.
\end{equation*}
Now, recall that the $\delta_i$ are the coefficients of the diagonal matrix $D_\varepsilon$ related to $C$ through the equation $CM_{\varepsilon}^T= P^{-1}D_\varepsilon P$, and that $X_C = \begin{pmatrix}
    A & C \\
    C^T & B
\end{pmatrix}$ is a covariance matrix. The condition of $X_C$ being positive-definite implies that the coefficients of $D_\varepsilon$ are $\delta_i^{+} =  \sqrt{r_i+ \frac{\varepsilon^2}{4}} -\frac{\varepsilon}{2}$. Finally, exploiting the relation $C_\varepsilon : = P^{-1} D_{\varepsilon}P (M_{\varepsilon}^T)^{-1} := A^{1/2}U^T D_\varepsilon U A^{-1/2} (M_{\varepsilon}^T)^{-1}$ we derive 

\begin{equation*}
    C_\varepsilon  = \left[A^{1/2}\left( A^{1/2}M_{\varepsilon}BM_{\varepsilon}^T A^{1/2} +\frac{\varepsilon^2}{4}I_d\right)^{1/2}A^{-1/2} - \frac{\varepsilon}{2}I_d \right](M_{\varepsilon}^T)^{-1},
\end{equation*}
that we write for short
\begin{equation*}
	C_\varepsilon  = \left[\left(A M_{\varepsilon}BM_{\varepsilon}^T  +\frac{\varepsilon^2}{4}I_d\right)^{1/2} - \frac{\varepsilon}{2}I_d \right](M_{\varepsilon}^T)^{-1}.
\end{equation*}
To conclude, one can check that the matrix $C_\varepsilon$ previously defined is such that $\nabla I_{\Sigma}^{\varepsilon}(C_\varepsilon) = 0$. As $I_{\Sigma}^{\varepsilon}$ is strictly convex on $\Pi^{++}(A, B)$ and differentiable on this domain, $C_\varepsilon$ is the unique minimizer of the objective function $\Pi^{++}(A,B)$. From Lemma \ref{lem:convexity_primal_obj}, $C_\varepsilon$ is also the unique minimizer of $I_\Sigma^{\varepsilon}$ on the set of admissible cross-covariance matrices $\Pi^{+}(A,B)$.

\end{proof}

We now flesh out the computations for deriving Corollary \ref{cor:closed_form_gauss}.

\begin{proof}
    As 
    \begin{equation*}
        W_{\Sigma}^{\varepsilon}(\mu, \nu) = \min_{C \in \Pi^{+}(A,B)} \left\langle Y + \varepsilon \Sigma^{-1}, X_C
         \right\rangle_{\HS} - \varepsilon \log \det (X_C), 
    \end{equation*}
    and we have derived the expression of $C_\varepsilon$ solution of this minimization problem, we can write 
    \begin{align*}
    W_{\Sigma}^{\varepsilon}(\mu, \nu)& =  \left\langle Y + \varepsilon \Sigma^{-1}, X_{C_\varepsilon}
         \right\rangle_{\HS} - \varepsilon \log \det (X_{C_\varepsilon}) \\
         & = \left\langle \begin{pmatrix}
        I_d & -I_d \\
        -I_d & I_d
        \end{pmatrix} + \varepsilon  \begin{pmatrix}
        \Gamma_{11} & \Gamma_{12} \\
        \Gamma_{12}^T& \Gamma_{22}
        \end{pmatrix}, \begin{pmatrix}
        A & C_\varepsilon \\
        C_\varepsilon^T& B
        \end{pmatrix}  \right\rangle_{\HS}  - \varepsilon \log \det \begin{pmatrix}
        A & C_\varepsilon \\
        C_\varepsilon^T& B
        \end{pmatrix} .
    \end{align*}
We begin by the scalar product and exploit the expression of $C_\varepsilon$ in Theorem \ref{thm:closed_form_gauss}. Recalling that the Hilbert-Schmidt scalar product is defined by the trace of the matrix product, and the notation $M_{\varepsilon}=I_d-\varepsilon \Gamma_{12}$, we derive 
\begin{align*}
    \left\langle Y + \varepsilon \Sigma^{-1}, X_{C_\varepsilon}
         \right\rangle_{\HS} & = \langle I_d + \varepsilon \Gamma_{11},A\rangle_{\HS} + \langle I_d + \varepsilon \Gamma_{22},B \rangle_{\HS} -2\tr(M_{\varepsilon}^TC_\varepsilon)
\end{align*}
We now focus on the term $\tr(M_{\varepsilon}^TC_\varepsilon)$ in last equation, and rewrite it as
\begin{align*}
    \tr( M_{\varepsilon}^T C_\varepsilon) & = \tr(C_\varepsilon M_{\varepsilon}^T) \\
    & = \tr\left(\left( AM_{\varepsilon}BM_{\varepsilon}^T+\frac{\varepsilon^2}{4}I_d\right)^{1/2} - \frac{\varepsilon}{2}I_d   \right)\\
    & =  \tr\left(  \left( AM_{\varepsilon}BM_{\varepsilon}^T+\frac{\varepsilon^2}{4}I_d\right)^{1/2}  \right) - \varepsilon\frac{d}{2}.
\end{align*}
Thus, we can rewrite the scalar product as
\begin{align}
    \left\langle Y + \varepsilon \Sigma^{-1}, X_{C_\varepsilon}
         \right\rangle_{\HS} = \tr(A) + \tr(B)& + \varepsilon  \langle \Gamma_{11}, A \rangle_{\HS} + \varepsilon \langle \Gamma_{22}, B \rangle_{\HS} \nonumber\\
         &- 2\tr\left(  \left( AM_{\varepsilon}BM_{\varepsilon}^T+\frac{\varepsilon^2}{4}I_d\right)^{1/2}  \right) + \varepsilon d.
\end{align}
We then turn to the $\log$-determinant term. For this computation, we will use the identity
\begin{equation*}
    \left( AM_{\varepsilon}BM_{\varepsilon}^T+ \frac{\varepsilon^2}{4} I_d\right)^{1/2} = A^{1/2}\left(A^{1/2}M_{\varepsilon}B M_{\varepsilon}^TA^{1/2}+ \frac{\varepsilon^2}{4}I_d \right)^{1/2}A^{-1/2}.
\end{equation*}
Thus, we can write the cross covariance block $C_\varepsilon$ as follows:
\begin{equation*}
    C_\varepsilon = A^{1/2} \left[\left(A^{1/2}M_{\varepsilon}B M_{\varepsilon}^TA^{1/2}+ \frac{\varepsilon^2}{4}I_d \right)^{1/2}- \frac{\varepsilon}{2}I_d   \right]A^{-1/2}(M_{\varepsilon}^T)^{-1}.
\end{equation*}
We will also exploit the determinant formula
$
    \det(X_{C_\varepsilon}) = \det(A)\det(B - C_\varepsilon^TA^{-1} C_\varepsilon).
$ 
We begin by computing 
\begin{align*}
C_\varepsilon^TA^{-1} C_\varepsilon & = M_\varepsilon^{-1}A^{-1/2}\left[\left(A^{1/2}M_{\varepsilon}B M_{\varepsilon}^TA^{1/2}+ \frac{\varepsilon^2}{4}I_d \right)^{1/2}- \frac{\varepsilon}{2}I_d   \right]^{2}A^{-1/2} (M_{\varepsilon}^T)^{-1} \\
& = M_\varepsilon^{-1}A^{-1/2}\left[A^{1/2}M_{\varepsilon}B M_{\varepsilon}^TA^{1/2} - \varepsilon \left(A^{1/2}M_{\varepsilon}B M_{\varepsilon}^TA^{1/2}+ \frac{\varepsilon^2}{4}I_d \right)^{1/2}+ \frac{\varepsilon^2}{2}I_d   \right]A^{-1/2} (M_{\varepsilon}^T)^{-1} \\
& = B+ M_{\varepsilon}^{-1}A^{-1/2}\left[\frac{\varepsilon^2}{2}I_d - \varepsilon \left(A^{1/2}M_{\varepsilon}B M_{\varepsilon}^TA^{1/2}+ \frac{\varepsilon^2}{4}I_d \right)^{1/2} \right]A^{-1/2} (M_{\varepsilon}^T)^{-1} 
\end{align*}
Next, 
\begin{align*}
    B - C_\varepsilon^TA^{-1} C_\varepsilon &=  M_{\varepsilon}^{-1}A^{-1/2}\left[\varepsilon \left(A^{1/2}M_{\varepsilon}B M_{\varepsilon}^TA^{1/2}+ \frac{\varepsilon^2}{4}I_d \right)^{1/2} -  \frac{\varepsilon^2}{2}I_d  \right]A^{-1/2} (M_{\varepsilon}^T)^{-1} \\
    & = \varepsilon M_{\varepsilon}^{-1}A^{-1}\left[ A^{1/2}\left(A^{1/2}M_{\varepsilon}B M_{\varepsilon}^TA^{1/2}+ \frac{\varepsilon^2}{4}I_d \right)^{1/2}A^{-1/2} - \frac{\varepsilon}{2}I_d\right] (M_{\varepsilon}^T)^{-1} \\
    & = \varepsilon M_{\varepsilon}^{-1}A^{-1}\left[ \left(AM_{\varepsilon}B M_{\varepsilon}^T+ \frac{\varepsilon^2}{4}I_d \right)^{1/2} - \frac{\varepsilon}{2}I_d \right] (M_{\varepsilon}^T)^{-1}.
\end{align*}
We can now compute the determinant of $X_{C_\varepsilon}$ as follows:
\begin{align*}
    \det(X_{C_\varepsilon}) & = \det(A)\det(B - C_\varepsilon^TA^{-1} C_\varepsilon) \\
    & = \det(A) \det\left(\varepsilon M_{\varepsilon}^{-1}A^{-1}\left[ \left(AM_{\varepsilon}B M_{\varepsilon}^T+ \frac{\varepsilon^2}{4}I_d \right)^{1/2} - \frac{\varepsilon}{2}I_d \right] (M_{\varepsilon}^T)^{-1}\right) \\
    & = \varepsilon^{d}\det(M_{\varepsilon})^{-2}\det\left( \left(AM_{\varepsilon}B M_{\varepsilon}^T+ \frac{\varepsilon^2}{4}I_d \right)^{1/2}- \frac{\varepsilon}{2}I_d\right).
\end{align*}
Taking the logarithm of last expression yields
\begin{equation}
    \log\det(X_{C_\varepsilon}) = d\log(\varepsilon) - 2\log\det(M_{\varepsilon}) + \log\det\left( \left(AM_{\varepsilon}B M_{\varepsilon}^T+ \frac{\varepsilon^2}{4}I_d \right)^{1/2}- \frac{\varepsilon}{2}I_d\right).
\end{equation}
Now, we will exploit the equality 
\begin{equation*}
    \left( \left(AM_{\varepsilon}B M_{\varepsilon}^T+ \frac{\varepsilon^2}{4}I_d \right)^{1/2}- \frac{\varepsilon}{2}I_d\right)^{-1} = \left( \left(AM_{\varepsilon}B M_{\varepsilon}^T+ \frac{\varepsilon^2}{4}I_d \right)^{1/2} +\frac{\varepsilon}{2}I_d\right) (AM_{\varepsilon}BM_{\varepsilon}^T)^{-1},
\end{equation*}
that derives from the identity
\begin{equation}
    \left( \left(AM_{\varepsilon}B M_{\varepsilon}^T+ \frac{\varepsilon^2}{4}I_d \right)^{1/2}- \frac{\varepsilon}{2}I_d\right) \left( \left(AM_{\varepsilon}B M_{\varepsilon}^T+ \frac{\varepsilon^2}{4}I_d \right)^{1/2} + \frac{\varepsilon}{2}I_d\right) = AM_{\varepsilon}B M_{\varepsilon}^T.
\end{equation}
From this equality we get 
\begin{align*}
     \log\det\left( \left(AM_{\varepsilon}B M_{\varepsilon}^T+ \frac{\varepsilon^2}{4}I_d \right)^{1/2}- \frac{\varepsilon}{2}I_d\right)& = -  \log\det\left( \left(AM_{\varepsilon}B M_{\varepsilon}^T+ \frac{\varepsilon^2}{4}I_d \right)^{1/2}+\frac{\varepsilon}{2}I_d\right)\\
     &  + \log \det(AB) + 2\log \det(M_{\varepsilon}).
\end{align*}
Thus, the $\log \det$ term of the optimal covariance matrix can be written
\begin{equation*}
    \log \det(X_{C_\varepsilon}) = d \log(\varepsilon) + \log \det(AB) -  \log\det\left( \left(AM_{\varepsilon}B M_{\varepsilon}^T+ \frac{\varepsilon^2}{4}I_d \right)^{1/2}+\frac{\varepsilon}{2}I_d\right).
\end{equation*}
Collecting the pieces of the previous computations, and recalling the additive constant $-\varepsilon(\log \det(\Sigma^{-1}) + 2d)$ we derive 
\begin{align*}
    W_\Sigma^{\varepsilon}(\mu, \nu) & = \tr(A) + \tr(B)  - 2\tr\left(  \left( AM_{\varepsilon}BM_{\varepsilon}^T+\frac{\varepsilon^2}{4}I_d\right)^{1/2}  \right) + \varepsilon \log  \det\left( \left(AM_{\varepsilon}B M_{\varepsilon}^T+ \frac{\varepsilon^2}{4}I_d \right)^{1/2}+ \frac{\varepsilon}{2}I_d\right)  \\
    & + \varepsilon  \left(\langle \Gamma_{11}, A \rangle_{\HS} + \ \langle \Gamma_{22}, B \rangle_{\HS} - \log\det(AB) - d-d\log(\varepsilon)  - \log \det(\Sigma^{-1}) \right).
\end{align*}

\end{proof}

\subsection{Dual problem approach}\label{sec:dual_problem}
In optimal transport problems, it is common practice to exploit tools from convex duality theory to characterize the sought solutions. In this section, we derive and solve the dual problem associated to the matrix reduction established in Lemma \ref{lem:regot_matrix}. To remain succinct, we only state the results and defer the related proofs in Section \ref{sec:proofs_dual} of the appendix. As in the previous section, $\Sigma \in S_{2d}^{++}(\mathbb{R})$ is a full-rank covariance matrix on $\mathbb{R}^d \times \mathbb{R}^d$ and $\varepsilon$ is a positive real number. We adapt the analysis of \citep[p.~26]{villani2021topics} to our framework. That is, when the measures $\mu = N_d(A)$ and $\nu = N_d(B)$ are centered Gaussian measures with full-rank covariance matrices $A$ and $B$.  We have shown in equation \eqref{eq:block_constrained_primal} that solving the optimal transport problem \eqref{eq:gaussian_kl_ot} associated to $W_\Sigma^\varepsilon(\mu, \nu)$ boils down to solving the matrix optimization problem
\begin{equation}\label{eq:dual_start_comput}
      \min_{X \in S_{2d}^+(\mathbb{R})}\langle Y + \varepsilon \Sigma^{-1}, X \rangle_{\rm HS} -  \varepsilon \log \det X \quad \text{such that} \quad X_{11} = A,~ X_{22} = B.
\end{equation}

\begin{prop}\label{prop:duality_OTGauss}
    Set $\mu = N_d(A)$ and $\nu = N_d(B)$ two Gaussian measures with full-rank covariance matrices $A$ and $B$. Introducing the matrix $M_\varepsilon$ defined in equation \eqref{eq:m_eps_matrix}, the entropic optimal transport problem \eqref{eq:gaussian_kl_ot} has dual formulation 
    \begin{align}\label{eq:duality_OTGauss}
        W_\Sigma^\varepsilon(\mu, \nu) =  \max_{(F,G)}  ~& \left\{\langle I_d-F, A \rangle_{\HS}  + \langle I_d -G, B \rangle_{\HS} 
          + \varepsilon \log  \det \begin{pmatrix}
         F    & -M_{\varepsilon}    \\
    -M_{\varepsilon}^T&    G 
    \end{pmatrix} \right\}\\ & \quad +  \varepsilon \left[ \langle \Gamma_{11}, A \rangle_{\HS} + \langle \Gamma_{22}, B \rangle_{\HS} -\log \det(\varepsilon \Sigma^{-1})\right], \nonumber
    \end{align}
where the maximum runs over the couples of $S_{d}^{++}(\mathbb{R})$.
\end{prop}
The proof of Proposition \ref{prop:duality_OTGauss} is deferred to Section \ref{sec:proofs_dual} of the appendix. It relies on standard tools from convex analysis. This Proposition \ref{prop:duality_OTGauss} shows an alternative optimization problem associated to our original optimal transport problem. From now on, we refer to the right hand side of equation \eqref{eq:duality_OTGauss} as to the dual problem.  The objective function $D_\Sigma^\varepsilon : S_d^{++}(\mathbb{R}) \times S_d^{++}(\mathbb{R}) \rightarrow \mathbb{R} \cup \{-\infty\}$ associated to this problem is called the dual function, and defined for every couple $(F,G) \in S_d^{++}(\mathbb{R}) \times S_d^{++}(\mathbb{R})$ by 
\begin{equation}\label{eq:dual_objective}
    D_{\Sigma}^\varepsilon(F,G) :=\langle I_d - F, A \rangle_{\HS} + \langle I_d-G, B \rangle_{\HS} +   \varepsilon \log  \det \begin{pmatrix}
      F   & -M_{\varepsilon}    \\
      -M_{\varepsilon} ^T  &    G 
    \end{pmatrix}.
\end{equation}

\begin{lem}The dual function \eqref{eq:dual_objective} is strictly concave on $\Pi(M_\varepsilon)$ the convex subset of $S_{d}^{++}(\mathbb{R}) \times S_{d}^{++}(\mathbb{R})$ defined by 
	\begin{equation}\label{eq:constraint_dual}
		\Pi(M_\varepsilon) := \left\{ (F, G) \in S_{d}^{++}(\mathbb{R}) \times S_{d}^{++}(\mathbb{R})~|~ \begin{pmatrix}
			F & -M_\varepsilon \\
			-M_\varepsilon^T & G
		\end{pmatrix} > 0 \right\}.
	\end{equation}
\end{lem}

\begin{proof}
	As the set of positive-definite matrices is convex, the set $\Pi(M_\varepsilon)$ is convex. Regarding the strict convexity of the dual function $D_\Sigma^{\varepsilon}$, up to additive constant, we can rewrite it as 
	\begin{equation*}
		D_\Sigma^{\varepsilon}(F,G) = \left\langle \begin{pmatrix} 
			F & -M_\varepsilon \\
			-M_{\varepsilon}^T & G
		\end{pmatrix}, \begin{pmatrix}
		- A & 0 \\
		0 & -B
	\end{pmatrix}  \right\rangle_{\HS}  +   \varepsilon \log  \det \begin{pmatrix}
	F   & -M_{\varepsilon}    \\
	-M_{\varepsilon} ^T  &    G 
\end{pmatrix} + \text{constants}.
	\end{equation*}
Exploiting the strict concavity of the log-determinant function on $S_{2d}^{++}(\mathbb{R})$ \citep[p.~42,~Cor.~1.4.2]{bakonyi2011matrix} allows to conclude that $D_\Sigma^{\varepsilon}$ is strictly concave on the set $\Pi(M_\varepsilon)$ that we introduced in equation \eqref{eq:constraint_dual}.
\end{proof} 

To detect the maximizer of the dual function \eqref{eq:dual_objective}, we study its first variation and its critical point.

\begin{prop}\label{prop:grad_dual}
    The dual function \eqref{eq:dual_objective} is differentiable at every $(F,G) \in S_d^{++}(\mathbb{R}) \times S_d^{++}(\mathbb{R})$ such that the matrix $G-M_\varepsilon F^{-1} M_\varepsilon^{T}$ is positive-definite. For such a couple $(F,G)$, denoting by $S = G-M_{\varepsilon}^{T}F^{-1}M_{\varepsilon}$ the Schur complement, the gradient of the dual function is given by the formula
    \begin{equation}\label{eq:grad_dual}
        \nabla D_{\Sigma}^\varepsilon(F,G) = (\varepsilon (F^{-1}+ F^{-1}M_{\varepsilon}S^{-1}M_{\varepsilon}^TF^{-1})-A, \varepsilon S^{-1}-B).
    \end{equation}
    Moreover, solving the gradient equation $\nabla D_{\Sigma}^\varepsilon(F,G) = (0,0)$ is equivalent to solving the matrix equations system
    \begin{equation}\label{eq:system_dual}
        \left\{ 
        \begin{array}{ccc}
            FAF - \varepsilon F - M_{\varepsilon}BM_{\varepsilon}^T& = & 0 \\
             \varepsilon B^{-1} + M_{\varepsilon}^TF^{-1}M_{\varepsilon}& = & G.
        \end{array}
        \right.
    \end{equation}
\end{prop}

The proof of Proposition \ref{prop:grad_dual} is deferred to Section \ref{sec:proofs_dual} of the appendix. Thanks to last proposition, we can find the solution to the dual problem by solving of a matrix-equation system. We can now give the solution to the variational representation \eqref{eq:duality_OTGauss} of entropic optimal transport.
\begin{thm}\label{thm:optim_potentials}
 If $\varepsilon >0$, the optimal dual variables $F_\varepsilon^{\star}, G_\varepsilon^{\star}$ associated to dual problem \eqref{eq:duality_OTGauss} are given by the formulae
 \begin{equation}\label{eq:optim_dual}
     \begin{array}{ccc}
        F_\varepsilon^{\star} &= &  A^{-1}\left( \frac{\varepsilon}{2}I_d + \left(AM_{\varepsilon}BM_{\varepsilon}^T+ \frac{\varepsilon^2}{4}I_d  \right)^{1/2} \right)\\
        G_\varepsilon^{\star} &=  & B^{-1}\left( \frac{\varepsilon}{2}I_d + \left(BM_\varepsilon^T AM_{\varepsilon}+ \frac{\varepsilon^2}{4}I_d  \right)^{1/2} \right).
     \end{array}
 \end{equation}
We can also express the second optimal dual variable $G_\varepsilon^\star$ as a function of $F_\varepsilon^\star$ through the relation
 \begin{equation}
      G_\varepsilon^{\star} =  \varepsilon B^{-1}+ M_\varepsilon^{T}(F_\varepsilon^{\star})^{-1}M_\varepsilon.
 \end{equation}
\end{thm}
Solving the system given in Proposition \ref{prop:grad_dual} is detailed in Section \ref{sec:proofs_dual} of the appendix. This is how we derive $(F_\varepsilon^{\star}, G_{\varepsilon}^{\star})$ in Theorem \ref{thm:optim_potentials}.
From this solution to the dual problem, we recover the regularized optimal transport cost already computed in Corollary \ref{cor:closed_form_gauss}.

\begin{cor}\label{cor:eot_cost_dual}
For $\mu = N_d(A)$ and $\nu = N_d(B)$ two centered Gaussian measures; the regularized optimal transport cost has the closed form expression given by the formula
\begin{align}
    W_{\Sigma}^{\varepsilon}(\mu, \nu)& = \tr(A) + \tr(B) - 2 \tr\left( \left(AM_{\varepsilon}BM_{\varepsilon}^T+ \frac{\varepsilon^2}{4}I_d \right)^{1/2} \right) + \varepsilon \log \det\left( \left(AM_{\varepsilon}BM_{\varepsilon}^T+ \frac{\varepsilon^2}{4}I_d \right)^{1/2}+\frac{\varepsilon}{2}I_d  \right)  \nonumber \\
   &+\varepsilon  \tr (\Gamma_{11} A)  +\varepsilon \tr(\Gamma_{22} B) - \varepsilon \log \det(AB) - \varepsilon d - \varepsilon d \log(\varepsilon) + \varepsilon\log \det(\Sigma).
\end{align}
\end{cor}

The computations leading to Corollary \ref{cor:eot_cost_dual} can be found in Section \ref{sec:proofs_dual} of the Appendix. The starting point is to plug $(F_{\varepsilon}^{\star}, G_{\varepsilon}^{\star})$, derived in Theorem \ref{thm:optim_potentials}, in dual problem \eqref{eq:duality_OTGauss}.

\section{Invertibility of $M_{\varepsilon}$ and examples of reference couplings}\label{sec:examples}

Our main result relies on the assumption that $M_{\varepsilon}$ is invertible. We now study the circumstances under which this holds.
First, we provide a probabilistic statement to the effect that the $\varepsilon$ for which  $M_{\varepsilon}$ is singular belong to a subset of probability zero. In that sense, random choice of $\varepsilon$ according to a continuous distribution guarantees almost sure invertibility.	
In addition to this, we state deterministic bounds on the value of $\varepsilon$ that guarantee invertibility. We end this section by introducing in Subsection \ref{sec:indep_coord} a class of reference couplings where the matrix $M_{\varepsilon}$ is automatically invertible.

\subsection{Invertibility of $M_{\varepsilon}$}\label{sec:inv_meps}

\paragraph{Probabilistic choice.}
As shown by next result,  $M_{\varepsilon}$ is \emph{generically} invertible. 
	\begin{lem}\label{lem:meps_invertible_surely}
		Let $\Sigma$ be a positive-definite covariance matrix acting on $\mathbb{R}^d \times \mathbb{R}^d$ and denote by $\Gamma_{12}$ the $d\times d$ upper-right block of its inverse. Suppose that $\varepsilon$ is a random variable over $\mathbb{R}_+$ whose distribution is absolutely continuous with respect to  Lebesgue measure. In such a case,
		\begin{equation}
			M_{\varepsilon} = I_d -\varepsilon \Gamma_{12}
		\end{equation}
		is almost surely invertible.
	\end{lem}
	\begin{proof}
		For every $\varepsilon > 0$, we have the equivalences
		\begin{align*}
			\det(M_{\varepsilon}) =0  & \Leftrightarrow  \det(I_d - \varepsilon \Gamma_{12}) = 0 \\
			& \Leftrightarrow (-\varepsilon)^d \det(\Gamma_{12} - \varepsilon^{-1}I_d) = 0 \\
			& \Leftrightarrow  \det(\Gamma_{12} - \varepsilon^{-1}I_d) = 0.	
		\end{align*}
		From the last equality $M_{\epsilon}$ is not invertible if and only if $\varepsilon^{-1}$ is a root of the characteristic polynomial of the matrix $\Gamma_{12}$. Recalling that $\Gamma_{12}$ is of dimension $d\times d$, its characteristic polynomial has at most $d$ real roots. As $\varepsilon$ is assumed to have a density over $\mathbb{R}_+$, denoting by $\sigma(\Gamma_{12})$ the finite spectrum of $\Gamma_{12}$, the probability of the event $\{\varepsilon \in \sigma(\Gamma_{12}) \}$ is null. Therefore, the event $\det(M_{\varepsilon}) =0 $ has zero probability; $M_{\epsilon}$ is almost-surely invertible.
	\end{proof}

\paragraph{Deterministic sufficient condition.} To give our deterministic argument, we introduce the following block decomposition of the reference matrix:
\begin{equation}\label{eq:ref_cov_blocks}
	\Sigma = \begin{pmatrix}
		A_{\refe} & C_{\refe} \\
		C_{\refe}^T & B_{\refe}
	\end{pmatrix}.
\end{equation}
The blocks $A_{\refe}, B_{\refe}$ and $C_{\refe}$ are squared matrices of dimension $d \times d$ with $A_{\refe}$ and $ B_{\refe}$ positive-definite. A classical result of Baker \cite[Thm.1.A]{baker1973joint} ensures the existence of a matrix $R_{\refe}$ of matrix norm at most 1 such that 
\begin{equation}\label{eq:ref_cov_correlation}
C_{\refe} = A_{\refe}^{1/2} R_{\refe} B_{\refe}^{1/2}.
\end{equation}
In our case, as $\Sigma$ is assumed to have full rank, the inequality $\|R_{\refe}\|_{\op} < 1$ holds true. Thanks to this block decomposition, and exploiting Lemma \ref{lem:block_inv}, we can rewrite $M_{\varepsilon}$ as
\begin{equation}\label{eq:block_inv_cov}
		M_\varepsilon = I_d + \varepsilon A^{-1}_{\refe}C_{\refe}(B_{\refe} - C_{\refe}^TA_{\refe}^{-1}C_{\refe})^{-1}.
\end{equation}
In the following proof, we apply the singular value decomposition to $R_{\refe}$, and the spectral theorem to the blocks $A_{\refe}$ and $B_{\refe}$. We remind that from these theorems, there exist $(\sigma_i[r], e_i[r], f_i[r])_{1 \leq i \leq d} \subset \mathbb{R}_+ \times \mathbb{R}^d \times \mathbb{R}^d$, $(\lambda_i[a], e_i[a])_{1 \leq i \leq d}  \subset \mathbb{R}_+ \times \mathbb{R}^d $, and $(\lambda_i[b], e_i[b])_{1 \leq i \leq d}  \subset \mathbb{R}_+ \times \mathbb{R}^d $ such that for every $x \in \mathbb{R}^d$, the equalities
\begin{equation}\label{eq:decompo_blocks_ref}
	R_{\refe} x = \sum_{i=1}^d \sigma_i[r] \langle f_i[r],x \rangle e_i[r], ~ 	A_{\refe} x = \sum_{i=1}^d \lambda_i[a] \langle e_i[a],x \rangle e_i[a] ~ \text{and} ~ B_{\refe} x = \sum_{i=1}^d \lambda_i[b] \langle e_i[b],x \rangle e_i[b]
\end{equation}
hold true. In the previous decompositions, the singular values $(\sigma_i[r])_{1 \leq i \leq d}$, and the spectral values $(\lambda_i[a])_{1 \leq i \leq d}$ and $(\lambda_i[b])_{1 \leq i \leq d}$ are ordered decreasingly. With this convention, $\sigma_1[r]$ is the largest singular value of $R_{\refe}$, and $\lambda_d[a]$ and $ \lambda_d[b]$ are the smallest eigenvalues of $A_{\refe}$ and $B_{\refe}$.

\begin{lem}\label{lem:invertibility_meps}
	Set $\Sigma \in S_{2d}^{++}(\mathbb{R})$ and $\varepsilon > 0$ and let $A_{\refe}$, $B_{\refe}$ and $R_{\refe}$ be the same blocks as in factorization \eqref{eq:ref_cov_correlation} of the reference cross-covariance matrix. If the inequality  
	\begin{equation}\label{eq:invertible_mat_opnorm}
\varepsilon \left(\frac{   \| R_{\refe}\|_{\op}}{1- \|R_{\refe}\|_{\op}^2} \right)< \frac{1}{\|A_{\refe}^{-1/2}\|_{\op}\|B_{\refe}^{-1/2}\|_{\op}}
	\end{equation}
	holds true, the matrix $M_\varepsilon = I_d-\varepsilon\Gamma_{12}$ is invertible. The last inequality can be formulated with the eigenvalues and the singular values. Denoting by  $\sigma_1[r]$ the largest singular value of $R_{\refe}$, and by $\lambda_d[a]$ and $ \lambda_d[b]$ the smallest eigenvalues of $A_{\refe}$ and $B_{\refe}$, the inequality 
	\begin{equation}\label{eq:inequality_invertible_mat}
		\varepsilon \left(\frac{\sigma_1[r]}{1 -\sigma_1[r]^2} \right) < \sqrt{\lambda_d[a] \lambda_d[b]} 
	\end{equation}
	ensures the invertibility of $M_{\varepsilon}$.
\end{lem}

\begin{proof}
	We will show that $\varepsilon \Gamma_{12}$ has matrix norm less than one. From equation \ref{eq:block_inv_cov}, we can express $\Gamma_{12}$ as
	\begin{equation*}
		\Gamma_{12} = -A^{-1}_{\refe}C_{\refe}(B_{\refe} - C_{\refe}^TA_{\refe}^{-1}C_{\refe})^{-1}
	\end{equation*} 
	Let us introduce $R_{\refe}$ the correlation matrix such that $C_{\refe} = A_{\refe}^{1/2} R_{\refe} B_{\refe}^{1/2}$. From this expression of the cross-covariance matrix, we derive the equalities
	\begin{align*}
		\varepsilon \Gamma_{12} &=  \varepsilon A_{\refe}^{-1/2} R_{\refe}B_{\refe}^{1/2}(B_{\refe}- B_{\refe}^{1/2} R_{\refe}^T R_{\refe} B_{\refe}^{1/2})^{-1} \\
		& =  \varepsilon A_{\refe}^{-1/2} R_{\refe}(I_d- R_{\refe}^T R_{\refe})^{-1} B_{\refe}^{-1/2}.
	\end{align*}
	Applying the matrix norm on both sides of the equality yields
	\begin{align}\label{eq:upper_bound_norm_gamma}
		\|\varepsilon \Gamma_{12}\|_{\rm op}& = \varepsilon \|A_{\refe}^{-1/2} R_{\refe}(I_d- R_{\refe}^T R_{\refe})^{-1} B_{\refe}^{-1/2}\|_{\op} \nonumber\\
		& \leq \varepsilon  \|A_{\refe}^{-1/2}\|_{\op} \| R_{\refe}\|_{\op} \|(I_d- R_{\refe}^T R_{\refe})^{-1}\|_{\op} \|B_{\refe}^{-1/2}\|_{\op}.
	\end{align}
	With the singular value decomposition of $R_{\refe}$, and following the same notations than in equation \eqref{eq:decompo_blocks_ref}, we have for every $x \in \mathbb{R}^d$ that
	\begin{equation}
		R_{\refe} x = \sum_{i=1}^d \sigma_i[r] \langle f_i[r],x \rangle e_i[r].
	\end{equation}
	As $R_{\refe}$ has matrix norm less than one, every singular value $\sigma_i[r]$ is smaller than one. Having arranged the singular values decreasingly, we have that $\|R_{\refe}\|_{\op}= \sigma_1[r]$.
	The singular value decomposition of $R_{\refe}$ implies the following spectral decomposition for $I_d - R_{\refe}^T R_{\refe}$: for every $x \in \mathbb{R}^d$
	\begin{equation}
		(I_d - R_{\refe}^T R_{\refe})x = \sum_{i=1}^d (1 - \sigma_i^2[r]) \langle f_i[r], x \rangle f_i[r].
	\end{equation}
	The singular values $\sigma_i[r]$ being smaller than one, we deduce that $I_d - R_{\refe}^T R_{\refe}$ is invertible and that its inverse has the explicit expression
	\begin{equation}
		(I_d - R_{\refe}^T R_{\refe})^{-1} =  \sum_{i=1}^d \frac{1}{1 - \sigma_i^2[r]}  f_i[r] f_i[r]^T.
	\end{equation}
	This is the spectral decomposition of $(I_d - R_{\refe}^T R_{\refe})^{-1}$. We deduce from it
	\begin{equation*}
	\|(I_d- R_{\refe}^T R_{\refe})^{-1}\|_{\op} = \frac{1}{1-\sigma_1^{2}[r]} = \frac{1}{1- \|R_{\refe}\|_{\op}^2}.
	\end{equation*}
	Going back to inequality \eqref{eq:upper_bound_norm_gamma}, and reminding last equality we derive
	\begin{equation*}
		\|\varepsilon\Gamma_{12}\|_{\op} \leq \frac{ \varepsilon \|A_{\refe}^{-1/2}\|_{\op} \| R_{\refe}\|_{\op} \|B_{\refe}^{-1/2}\|_{\op}}{1- \|R_{\refe}\|_{\op}^2} < 1
	\end{equation*}
thanks to inequality \eqref{eq:invertible_mat_opnorm}. This ensures invertibility of $M_{\varepsilon} = I_d - \varepsilon \Gamma_{12}$. To show that inequality \eqref{eq:inequality_invertible_mat} implies the invertibility of $M_{\varepsilon}$, first remind that
\begin{equation*}
	\| R_{\refe} \|_{\op} = \sigma_1[r].
\end{equation*}
And from the spectral decompositions 
\begin{equation*}
	A_{\refe} x = \sum_{i=1}^d \lambda_i[a] \langle e_i[a],x \rangle e_i[a] \quad \text{and} \quad B_{\refe} x = \sum_{i=1}^d \lambda_i[b] \langle e_i[b],x \rangle e_i[b]
\end{equation*}
we derive that $\|A_{\refe}^{-1/2}\|_{\op}= 1/\sqrt{\lambda_d[a]}$ and $\|B_{\refe}^{-1/2}\|_{\op}= 1/\sqrt{\lambda_d[b]}$. Substituting the operator norms by these values in \eqref{eq:invertible_mat_opnorm} yields \eqref{eq:inequality_invertible_mat}.
\end{proof}

\subsection{Reference plan parametrized by a correlation matrix}\label{sec:correlation_parameter}

So far, we have addressed the case where the Gaussian prior has arbitrary covariance matrix $\Sigma$. However, the constraint set $\Pi(\mu, \nu)$ imposes that the solution to our Gaussian problem has a covariance with diagonal blocks $A$ and $B$. In this section we study the case where the prior covariance matrix also has diagonal blocks $A$ and $B$. In this case, it only remains to choose the cross-covariance matrix $C$. However, every valid cross covariance matrix decomposes as $C=A^{1/2} R_{\refe} B^{1/2}$, with $R_{\refe}$ a correlation matrix with matrix norm $\| R_{\refe} \|_{\op} \leq 1$. Thus, the reference coupling has covariance matrix
\begin{equation}\label{eq:prior_cov_cor_ot}
	\Sigma = \begin{pmatrix}
		A &  A^{1/2} R_{\refe} B^{1/2}\\
		B^{1/2}R_{\refe}^TA^{1/2}& B
	\end{pmatrix}.
\end{equation}
As $\Sigma$ is assumed invertible, $R_{\refe}$ is such that $\|R_{\refe}\|_{\op} < 1$. Applying Theorem \ref{thm:closed_form_gauss}, we derive the solution of entropic optimal transport \eqref{eq:gaussian_kl_ot} when the reference covariance is of the form \eqref{eq:prior_cov_cor_ot}.
\begin{cor}\label{cor:solution_correlation_parameter}Set $\mu=N_d(A)$ and $\nu=N_d(B)$ two centered Gaussian measures with full-rank covariance matrices $A$ and $B$. For $\varepsilon > 0$ and $R_{\refe} \in M_d(\mathbb{R})$ a correlation matrix with matrix norm smaller than one such that Assumption \ref{hyp:invertible_mat} holds true, the entropic transport problem 
	\begin{equation}\label{eq:prior_cor_eot}
		\min_{\pi \in \Pi(\mu, \nu)} \int_{\mathbb{R}^d \times \mathbb{R}^d} \|x-y\|^2 d\pi(x,y) + 2\varepsilon \KL(\pi|N_{2d}(\Sigma)) \quad \text{with} \quad \Sigma=\begin{pmatrix}
			A &  A^{1/2} R_{\refe} B^{1/2}\\
			B^{1/2}R_{\refe}^TA^{1/2}& B
		\end{pmatrix},
	\end{equation}
	has solution
	\begin{equation*}
		N_{2d}\begin{pmatrix}
			A & C_\varepsilon\\
			C_\varepsilon^T& B
		\end{pmatrix}, 
	\end{equation*}
	where 
	\begin{equation}\label{eq:sol_correlation}
		C_\varepsilon =A^{1/2}\left[\left( NN^T+\frac{\varepsilon^2}{4}I_d\right)^{1/2} - \frac{\varepsilon}{2}I_d \right](N^T)^{-1}B^{1/2}, \quad \text{and} \quad N = A^{1/2}B^{1/2} + \varepsilon R_{\refe}(I_d - R_{\refe}^T R_{\refe})^{-1}.
	\end{equation}
\end{cor}

\begin{proof}
	We now detail the computations that led to the closed form \eqref{eq:sol_correlation} for the cross covariance $C_\varepsilon$.
	Substituting $A_{\refe}$ by $A$,  $B_{\refe}$ by $B$, and $C_{\refe}$ by $A^{1/2} R_{\refe}B^{1/2}$ in formula \eqref{eq:block_inv_cov}, the matrix $M_{\varepsilon}= I_d - \varepsilon \Gamma_{12}$ that appears in Theorem \ref{thm:closed_form_gauss} is now given by 
	\begin{equation*}
		M_{\varepsilon}= I_d + \varepsilon A^{-1/2} R_{\refe}(I_d - R_{\refe}^TR_{\refe})^{-1}B^{-1/2}.
	\end{equation*}
	With the purpose of circumventing intricate expressions for $C_\varepsilon$, we factorize $M_\varepsilon$ as 
	\begin{align*}
		M_{\varepsilon}& = A^{-1/2}(A^{1/2}B^{1/2}+ \varepsilon R_{\refe}(I_d - R_{\refe}^T R_{\refe})^{-1})B^{-1/2}\\
		& = A^{-1/2} N B^{-1/2}, 
	\end{align*}
	where $N=A^{1/2}B^{1/2}+ \varepsilon R_{\refe}(I_d - R_{\refe}^TR_{\refe})^{-1}$. Now, from Theorem \ref{thm:closed_form_gauss}, we know that the cross covariance matrix is given by the formula 
	\begin{align*}
		C_\varepsilon &=\left[\left( AM_{\varepsilon}BM_{\varepsilon}^T+\frac{\varepsilon^2}{4}I_d\right)^{1/2} - \frac{\varepsilon}{2}I_d \right](M_\varepsilon^T)^{-1} \\
		& = A^{1/2}\left[ \left( A^{1/2}M_{\varepsilon}BM_\varepsilon^TA^{1/2} +\frac{\varepsilon^2}{4}I_d\right)^{1/2}- \frac{\varepsilon}{2}I_d \right]A^{-1/2} (M_\varepsilon^T)^{-1}.
	\end{align*}
	Next, we observe the simplification
	\begin{equation}
		A^{1/2}M_{\varepsilon}BM_\varepsilon^T A^{1/2} = NN^T.
	\end{equation}
	From this last observation, we finally reach the expression
	\begin{equation*}
		C_{\varepsilon} = A^{1/2}\left[ \left(N N^T+\frac{\varepsilon^2}{4}I_d\right)^{1/2}- \frac{\varepsilon}{2}I_d \right] (N^T)^{-1}B^{1/2}.
	\end{equation*}
	
\end{proof}

\subsection{Independent-coordinates reference coupling}\label{sec:indep_coord}
We now introduce a class of coupling covariances that ensures invertibility of the matrix $M_{\varepsilon}$. We call independent-coordinate covariance a matrix $\Sigma_{\rho} \in S_{2d}^{++}(\mathbb{R})$  of the form 
\begin{equation}\label{eq:indep_coord_couple}
	\Sigma_{\rho} = \begin{pmatrix}
		I_d & \diag(\rho_1,\ldots, \rho_d) \\
		\diag(\rho_1,\ldots, \rho_d) & I_d
	\end{pmatrix} \quad \text{where} \quad \forall i \in \{1,\ldots,d\},~0 \leq \rho_i< 1
\end{equation}
and $\diag(\rho_1,\ldots, \rho_d)$ is the $d \times d$ diagonal matrix with diagonal vector $(\rho_1,\ldots, \rho_d) \in \mathbb{R}^d$. If $(Z_1, Z_2) \in \mathbb{R}^d \times \mathbb{R}^d$ is a Gaussian couple where $Z_1$ and $Z_2$ have their own coordinates independent, up to rescaling, its covariance is of the form \eqref{eq:indep_coord_couple}. That is why we call such matrices independent-coordinate reference couplings. As $\rho_i \in [0,1)$, such a matrix is invertible and has inverse 
\begin{equation}
	\Sigma_{\rho}^{-1} =  \begin{pmatrix}
		\diag\left(\frac{1}{1-\rho_1^2}, \ldots, \frac{1}{1-\rho_d^2}\right) & -	\diag\left(\frac{\rho_1}{1-\rho_1^2}, \ldots, \frac{\rho_d}{1-\rho_d^2}\right) \\
	-	\diag\left(\frac{\rho_1}{1-\rho_1^2}, \ldots, \frac{\rho_d}{1-\rho_d^2}\right) &	\diag\left(\frac{1}{1-\rho_1^2}, \ldots, \frac{1}{1-\rho_d^2}\right)
	\end{pmatrix}.
\end{equation} 
In this case, the matrix $M_{\varepsilon} = I_d -\varepsilon \Gamma_{12}$ that appears in our main Theorem \ref{thm:closed_form_gauss} simplifies to 
\begin{equation}
	M_\varepsilon = \diag\left(1+\frac{\varepsilon \rho_1}{1-\rho_1^2}, \ldots, 1+\frac{\varepsilon \rho_d}{1-\rho_d^2}\right),
\end{equation}
which is invertible for any value of $\varepsilon > 0$. We now study the scenario where all correlation coefficients $\rho_1,\ldots,\rho_d$ are equal to the same $\rho \in [0,1)$. As we will this in the next result, this choice of reference coupling connects to entropic optimal transport with product measure as reference coupling.
 
\begin{cor}
	Set $\varepsilon > 0$. Let $N_d(A)$ and $N_d(B)$ be two Gaussian measures and consider an independent coordinate coupling $\Sigma_{\rho}$ with correlation parameter $\rho\in[0,1)$, that is 
	\begin{equation}
	\Sigma_{\rho}=\begin{pmatrix}
		I_d & \rho I_d \\
		\rho I_d & I_d
		\end{pmatrix}.
	\end{equation}
In this case, the cross correlation matrix $C_{\varepsilon}$ solution of the entropic optimal transport problem reduces to
\begin{equation}\label{eq:sol_indep_coord}
	C_{\varepsilon} = \left[A^{1/2}\left( A^{1/2}B A^{1/2} +\frac{\varepsilon(\rho)^2}{4}I_d\right)^{1/2}A^{-1/2} - \frac{\varepsilon(\rho)}{2}I_d\right] \quad \text{with} \quad \varepsilon(\rho) := \varepsilon\left(1+\frac{\varepsilon \rho}{1-\rho^2} \right)^{-1}.
\end{equation}
Moreover, we have the following asymptotic behaviors for $\varepsilon(\rho)$:
\begin{equation}\label{eq:asymptotic_rho_eps}
\varepsilon(\rho)  \underset{\rho \rightarrow 0^{+}}{\sim} \varepsilon \qquad \text{and} \qquad \varepsilon(\rho)  \underset{\rho \rightarrow 1^{-}}{\sim} 2(1-\rho).
\end{equation}

\end{cor}

Before proving this last result, we make precise how it connects to standard entropic optimal transport. The cross-correlation matrix \eqref{eq:sol_indep_coord} is exactly the solution of entropic optimal transport with product measure as reference and regularization parameter $\varepsilon(\rho)$ given in \eqref{eq:sol_indep_coord}. While this result may appear like a return to the product measure, we believe that it gives an alternative interpretation of entropic optimal transport. Penalizing the optimal transport problem by adding $2\varepsilon \KL(\cdot | \mu \otimes \nu)$ is equivalent, when $\varepsilon$ goes to zero, to the addition of the penalty term
\begin{equation}
	2 \KL(\cdot | N(\Sigma_{\varepsilon})) \quad \text{with} \quad \Sigma_{\varepsilon} = \begin{pmatrix}
		I_d & (1-\varepsilon/2)I_d \\
		(1-\varepsilon/2)I_d & I_d
	\end{pmatrix}.
\end{equation}
In this alternative interpretation, the rate of epsilon going to zero encodes the correlation parameter of the reference coupling in the penalty term. This observation will reveal valuable in the next section. The two measures $\mu$ and $\nu$ will be two time-marginals of the same Gaussian process, with small time gap. In this scenario, a smaller time gap should lead to larger correlation coefficient in the reference coupling. Before moving to our application to trajectory sampling, we point out that if the $\rho_i$ are not chosen all equal, the equivalence with the reference product measure does not hold any more.

\begin{proof}
	Formula \eqref{eq:sol_indep_coord} is a consequence of Theorem \ref{thm:closed_form_gauss} in the case where $M_{\varepsilon}$ reduces to $M_{\varepsilon}= (1+\varepsilon\rho/(1-\rho^2))I_d$. Regarding the asymptotic behavior of $\varepsilon(\rho)$, we focus on the regime where $\rho$ converges increasingly toward one in equation \eqref{eq:asymptotic_rho_eps}. For this purpose, we write 
\begin{align*}
	\varepsilon(\rho)& = \varepsilon \left(\frac{1-\rho^{2} + \varepsilon \rho}{(1-\rho)(1+\rho)}\right)^{-1} \\
	& = \varepsilon(1-\rho)\left(\frac{1-\rho^{2} + \varepsilon \rho}{(1+\rho)}\right)^{-1}\\
	& = 2(1-\rho)\left(\frac{2\varepsilon^{-1}(1-\rho^{2}) + 2\rho}{(1+\rho)}\right)^{-1}.
\end{align*}
Then, one can check that for every value of $\varepsilon$, the last factor in last equality converges toward one when $\rho$ goes to one. This means 
\begin{equation*}
	\varepsilon(\rho)  \underset{\rho \rightarrow 1^{-}}{\sim} 2(1-\rho),
\end{equation*}
as claimed.
\end{proof}

\section{Trajectory Reconstruction: From Statics to Dynamics}\label{sec:static_to_dyn}

\subsection{Framework and sampling algorithm}
We now assume to have $\mu_{t_1},\ldots, \mu_{t_{n}}$ a finite collection of Gaussians on $\mathbb{R}^d$, interpreted as the time marginals of some continuous-time process in $d$-dimensions (or, alternatively, of an interacting particle system comprised of $d$ particles, each evolving in $\mathbb{R})$. {
{Importantly, this reduction is only meaningful if the reference structure used in each pairwise problem is consistent with a common underlying process. A continuous-time Gaussian reference process provides exactly this consistency, while still allowing each pairwise problem to be solved independently}. Note that if we are observing marginals at an increasingly dense collection of time points in a compact time interval (say [0,1]), the only reference process consistent with a product reference at all scales is a coloured noise process -- not even well defined as a process. Thus, within the framework of trajectory inference, product couplings are ill-suited as references, as they steer toward independent transitions at any scale -- no matter how local-- which is at odds with the very temporal coherence of the process. \\

Concretely, let $(X_t^{\obs})_{t \in [0,1]}$ be a centered Gaussian process on $\mathbb{R}^d$ and $0\le t_1 < t_2 < \cdots < t_n \le 1$ be $n$ observation times. Suppose that at each time $t_j$ we observe (or estimate) the marginal
\begin{equation}\label{eq:marginals}
\mu_{t_j} := \mathcal{L}(X_{t_j}^{\obs}).
\end{equation}
As $(X_t^{\obs})_{t \in [0,1]}$ is supposed to be centered and Gaussian, for any $j \in \{1,\ldots, n\}$, we can write $\mu_{t_j}  = N(A_j)$ where $A_{j}\in S_d^{++}(\mathbb{R})$ is symmetric positive definite.  These Gaussian measures $N(A_1), \ldots, N(A_n)$ are the static inputs of our problem. To induce dynamics, we choose an other centered Gaussian process $(Z_t)_{t \in [0,1]}$ that controls the reconstructed trajectory. Assuming the Gaussian process $(Z_t)_{t \in [0,1]}$ centered, it is completely characterized by its matrix-valued covariance kernel $K_Z : [0,1]^2 \rightarrow M_d(\mathbb{R})$ defined at every $s,t$ by 
\begin{equation}\label{eq:reference_kernel_proc}
	K_Z(s,t) = \mathbb{E}[Z_s Z_t^T].
\end{equation}
At any pair of successive times $(t_j, t_{j+1})$ this kernel \eqref{eq:reference_kernel_proc} yields the Gaussian reference coupling
\begin{equation}\label{eq:reference_covar_proc}
 N(\Sigma_j)
\quad \text{where} \quad
\Sigma_j=
\begin{pmatrix}
K_Z(t_j,t_j)  & K_Z(t_j,t_{j+1})\\
K_Z(t_j,t_{j+1})^T & K_Z(t_{j+1},t_{j+1})
\end{pmatrix}\in S_{2d}^{++}(\mathbb{R}).
\end{equation}
In other words, the reference Gaussian coupling $N(\Sigma_j)$ is the law of the 2$d$-vector $(Z_{t_j}, Z_{t_{j+1}}) \in \mathbb{R}^d \times \mathbb{R}^d$. We point out that the family $\{\Sigma_j\}_{j=1}^{n-1}$ is automatically \emph{consistent across time} because it is obtained from one underlying continuous-time process. We now define a local objective that decomposes into successive pairs, and thus is amenable to our previous analysis.
Define the induced dynamics on the grid $\{t_j\}$ by solving, for each step independently,
the entropic optimal transport problem with a \emph{non-product} Gaussian reference:
\begin{equation}\label{eq:time_entropic_ot_ref}
N(\Sigma_{j}^{\varepsilon}) = \argmin_{\pi\in\Pi(\mu_{t_{j}},\mu_{t_{j+1}})}
\left\{
\int_{\mathbb{R}^d\times \mathbb{R}^d}\|x-y\|^2 \diffd\pi(x,y)
\;+\;
2\varepsilon\,\KL\!\big(\pi |\,N(\Sigma_j)\big)
\right\},
\qquad j=1,\dots,n-1.
\end{equation}
This is exactly the Gaussian entropic OT problem we solved in Section \ref{sec:closed_form}, with the identifications
\[
A\leftarrow A_{j},\qquad 
B\leftarrow A_{j+1},\qquad 
\Sigma \leftarrow \Sigma_j.
\]
The point is that allowing general Gaussian reference couplings $N(\Sigma_j)$ (rather than $\mu_{t_j}\otimes\mu_{t_{j+1}}$)
introduces meaningful \emph{temporal structure} at each step while remaining analytically tractable.  From these couplings $(N(\Sigma_j^{\varepsilon}))_{1 \leq j \leq n-1}$ we can build a discrete time Markov chain $(\widehat{Z}_{t_j})_{1\leq j \leq n}$ such that for any time $t_j $, the couple $(\widehat{Z}_{t_j}, \widehat{Z}_{t_{j+1}})$ has distribution $N(\Sigma_j^{\varepsilon})$ the solution of \eqref{eq:time_entropic_ot_ref}. Writing the block decomposition 
\begin{equation}\label{eq:blockCovOptimalCoupling}
\Sigma_j^{\varepsilon} = 
\begin{pmatrix}
A_{j} & C_j\\
C_j^\top & A_{j+1}
\end{pmatrix},
\end{equation}
with $C_j \in M_d(\mathbb{R})$ given in Theorem \ref{thm:closed_form_gauss}, we can make the transition mechanism from time $t_{j}$ to time $t_{j+1}$ explicit:
\begin{equation}\label{eq:MarkovTransitionFromC}
\widehat{Z}_{t_{j+1}} =  C_j^\top A_{j}^{-1} \widehat{Z}_{t_j} + \eta_j,
\qquad
\eta_j\sim N(0,A_{j+1} - C_j^\top A_{j}^{-1} C_j),
\end{equation}
where $\eta_j$ is independent from all previous (w.r.t. the time index) random variables. 
From Theorem \ref{thm:closed_form_gauss}, we have an explicit formula for $C_{j}$. Thus, the Markov chain defined in \eqref{eq:MarkovTransitionFromC} can be sampled efficiently.\\

\begin{algorithm}[ht]
	\SetAlgoLined
	\KwIn{\textit{Observations}: marginal covariances $A_{1}, \cdots , A_{n}$ and times $t_1 < \ldots < t_n$\\
		\textit{Hyperparameters}: Matrix-valued kernel $K_{Z}$ and $\varepsilon \geq 0$}
	\textit{Initialization:} $\widehat{Z}_{t_1} \sim N(A_{1})$\\
	\For{$j \leftarrow 1$ \KwTo $n-1$}{
		\tcc{Compute the reference covariance $\Cov(Z_{t_j}, Z_{t_{j+1}})$}
		$\Sigma_{j} \gets \Cov(Z_{t_j}, Z_{t_{j+1}})$\\
		\tcc{Solve the optimal transport problem with reference coupling $N(\Sigma_j)$}
		$\begin{pmatrix}
			A_j & C_j \\
			C_j^{T} & A_{j+1} 
		\end{pmatrix} \gets $ solution of entropic optimal transport \eqref{eq:time_entropic_ot_ref}\\
		\tcc{Sample $\widehat{Z}_{t_{j+1}}$ from $\widehat{Z}_{t_j}$ }
		$\eta_j \sim N(A_{j+1}-C_j^TA_j^{-1}C_j)$\\
		$\widehat{Z}_{t_{j+1}} \gets  C_{j}^TA_{j}^{-1}\widehat{Z}_{t_j} + \eta_j$ 
	}
	
	\KwRet{$\left(\widehat{Z}_{t_1}, \ldots, \widehat{Z}_{t_n} \right)$}
	
	\caption{Dynamic induced by entropic optimal transport with reference process}
	\label{alg:sampling_markovian_eot}
\end{algorithm}

To summarize, choosing the global criterion as a \emph{sum of local entropic costs relative to the reference couplings}
yields a fully decoupled set of $n-1$ tractable pairwise problems \eqref{eq:time_entropic_ot_ref},
whose solutions can be glued into a coherent Gaussian Markov chain on $\{t_j\}_{j=1}^{n}$. An appealing feature of this approach is that it features a certain \emph{resolution invariance}: the inferred dynamics do not depend qualitatively on how finely time happens to be sampled, for example if intermediate time points are inserted or removed. By comparison, if one were to take $\mu_{t_{j}}\otimes\mu_{t_{j+1}}$ as the reference coupling for successive times, then one would steer toward independence between successive times.
In the present reconstruction viewpoint, this corresponds to a baseline transition kernel
\begin{equation*}
\mathbb{P}(Z_{t_{j+1}}\in\cdot\,\mid\,Z_{t_j}=x)=\mu_{t_{j+1}}(\cdot),
\end{equation*}
i.e.\ resampling from the next marginal regardless of the current state, which is a trivial (memoryless) notion of dynamics, in effect pure (colored) noise.
To encodes temporal coherence in a statistically meaningful way, one needs to use a correlated Gaussian reference coupling, which illustrates the importance of entropic OT with a general (correlated) reference. We further illustrate these points numerically in the next section.

\subsection{Numerical Examples}\label{sec:numerical_experiments}

We now numerically illustrate the framework of entropic OT with general Gaussian reference in the context of trajectory reconstruction, as laid out in the previous section. In our experimental set-up, the true dynamics  $(X^{\obs}_t)_{t\in[0,1]}$ on $\mathbb R^2$ arises from the linear diffusion 
\begin{equation}\label{eq:true_SDE}
	\diffd X_t^{\obs} = -K X_t^{\obs}\diffd t +  \diffd W_t, \quad \text{where} \quad K = \begin{pmatrix}
		3 & 0 \\
		2 & 3
	\end{pmatrix} 
\end{equation}
is the drift matrix and $W_t$ is a standard $2$-dimensional Brownian motion. Full trajectories from \eqref{eq:true_SDE} are displayed in Figure \ref{fig:3dim_sample_path_true_diff}. However, only static information at the times $t_1 < \ldots < t_n \subset [0,1]$ is available; namely, the time marginals $\mu_{t_j} = \mathcal{L}(X_{t_j}^{\obs})$.
In our experiments, the marginals admit the closed form given (see e.g. \cite[Prop.~3.5]{pavliotis2014stochastic}) by 
\begin{equation}\label{eq:marginal_experiments}
 \mu_{t_j} = N(A_j) \quad \text{where} \quad A_j = e^{-t_j K}\left(A_0 + \int_0^{t_j} e^{s(K+K^T)} \diffd s\right) e^{-t_jK^T}.
\end{equation}
In case a time marginal $\mu_{t_j}$ is unknown and only observable from samples, we would substitute $A_j$ by an estimator thereof. 
Equation \eqref{eq:marginal_experiments} encodes the static part of our framework.
Regarding the dynamic component, we need to choose a reference process $(Z_t)_{t \in [0,1]}$, or equivalently a covariance kernel as in equation \eqref{eq:reference_kernel_proc}. We take kernel matrices $K_{Z}$ corresponding to a process assumed to have independent coordinates. That yields reference couplings of the form introduced in Section \ref{sec:indep_coord}. And in this time-dependent scenario, they read
\begin{equation}\label{eq:indep_coord_kernel}
	K_Z(s,t) = \rho(s,t)I_d,
\end{equation}
for a scalar covariance kernel $\rho$. Consequently, the reference coupling $N(\Sigma_j)$ in \eqref{eq:reference_covar_proc} reduces to 
\begin{equation*}
	\Sigma_j = \begin{pmatrix}
		I_d & \rho(t_j,t_{j+1})I_d \\
		\rho(t_j,t_{j+1})I_d & I_d
	\end{pmatrix}.
\end{equation*}

We consider three choices for the kernel $\rho$: 

\begin{itemize}
\item The fractional Brownian Motion (fBM) kernel $\rho_H$ with parameter $H \in (0,1)$ defined, for $s \leq t$, by
\begin{equation}\label{eq:fbm_kernel}
\rho_H(s,t) = \frac{1}{2|t+1|^{2H}}\left(|t+1|^{2H} + |s+1|^{2H} - |t-s|^{2H}\right).
\end{equation}

\item The heat (Gaussian) kernel $\rho_\sigma$ with parameter $\sigma > 0$ defined, for $s,t \in [0,1]$, by
\begin{equation}\label{eq:heat_kernel}
\rho_{\sigma}(t,s) = \exp\left(-\frac{(t-s)^2}{2 \sigma^2}\right).
\end{equation}

\item The trivial (white noise) kernel,
\begin{equation}\label{eq:trivial_kernel}\rho_{\otimes}(s,t):={\bf 1}\{s=t\}.
\end{equation}
\end{itemize}
The fBM kernel corresponds to a continuous but non-differentiable Gaussian process. The parameter $H$ controls the Hölder regularity of the paths: smaller values of $H$ yield rougher trajectories, while larger values lead to smoother ones. The case $H=1/2$ yields standard Brownian motion. By contrast, the heat kernel is associated with highly regular (in fact, smooth) sample paths. Finally, the trivial kernel corresponds to Gaussian white noise, which possesses negative regularity (it is \emph{not} defined as a process but as a distribution) and corresponds to independent time marginals.

In each case, we will sample discrete-time processes $(\widehat{Z}_{t_j})_{1 \leq j \leq n}$ by way of  Algorithm \ref{alg:sampling_markovian_eot}. For visualization purposes, we interpolate linearly between successive realisations $\widehat{Z}_{t_{j}}$ and $\widehat{Z}_{t_{j+1}}$. 
 In conducting these experiments, we wish to primarily focus on the qualitative impact of the choice of reference kernel \eqref{eq:indep_coord_kernel} on the reconstructed trajectories. For this reason, we set $\varepsilon = 0.01$ throughout. We consider evenly spaced observation times $t_j=(j-1)/n$ for three different values $n \in \{100, 500, 1000\}$, corresponding to progressively finer time resolutions. 
 
 We begin by applying sampling Algorithm \ref{alg:sampling_markovian_eot} for classic entropic optimal transport, corresponding to choosing the trivial reference kernel $\rho_{\otimes}(s,t)$.  Figure \ref{fig:indep_ref_eot_paths} depicts the generated trajectory between time $0$ and time $1$, for marginals observed at evenly spaced times. In the low-resolution scenario $n=100$, the evolution appears plausible as a diffusion. However, as the time-resolution (and hence number of marginals) increases, the sampled trajectories feature increasingly erratic oscillations. This reflects that fact that the product reference cannot cannot accommodate temporal contiguity. 
  Next, we sample from Algorithm \ref{alg:sampling_markovian_eot}, with the fBM kernel of Hurst index $H=0.25$ as reference. Figure \ref{fig:fbm_ref_eot_paths} also shows a progression toward rougher trajectories as $n$ grows, but without degenerating into white noise. Increasing the Hurst parameter to $H=1/2$ (Brownian motion reference) yields the sample paths displayed in Figure \ref{fig:brown_ref_eot_paths}. Qualitatively, the paths feature the regularity one expects of a diffusion driven by Brownian motion. Choosing the heat kernel as a reference, corresponds to highly regular reference paths -- correspondingly, one observes in Figure \ref{fig:heat_ref_eot_paths} that the sampled trajectories arae very smooth and exhibit  minimal oscillations. This seems especially true in the highest resolution regime when $n=1000$.\\

The code to reproduce the experiments is available at \url{https://github.com/Paul-Freulon/Entropic_Optimal_Transport_Reference_Coupling}.

\begin{figure}[h]
	\centering
	\includegraphics[width=0.6\linewidth,]{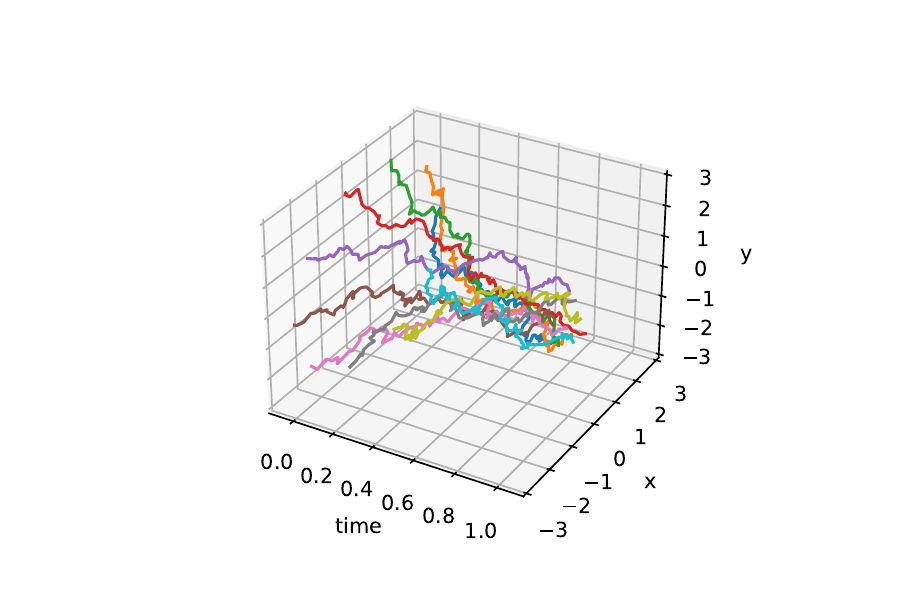}
	\caption{Paths simulated from linear diffusion \eqref{eq:true_SDE}}
	\label{fig:3dim_sample_path_true_diff}
	
\end{figure}

\begin{figure}[!h]

	\includegraphics[width=1\linewidth,]{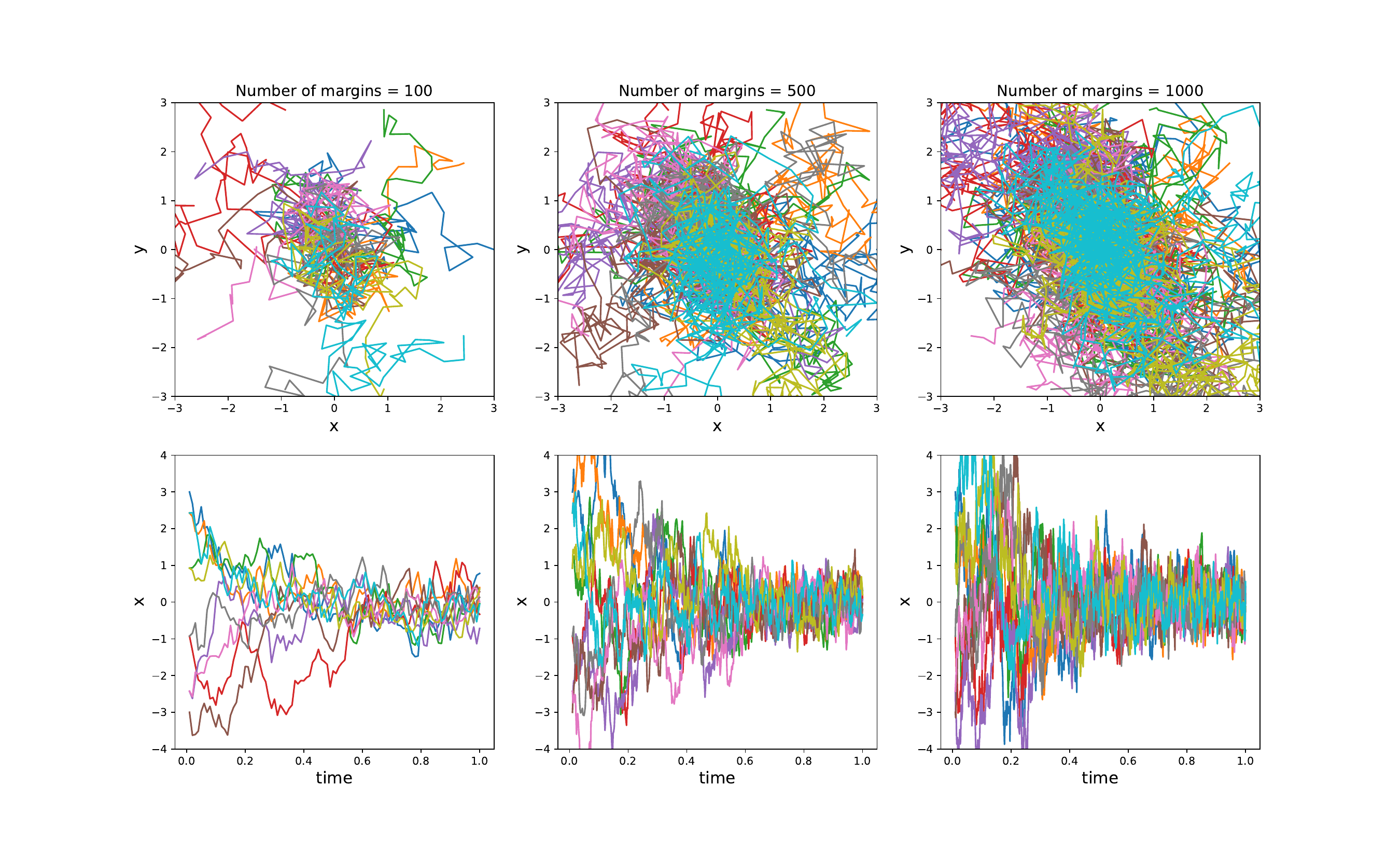}
	\caption{Sample path reconstructed with entropic optimal transport with {\bf independent coupling} reference. Top panels: two dimensional projection of the full trajectories. Bottom panels: evolution over time of the $x$-axis.}
	\label{fig:indep_ref_eot_paths}
\end{figure}
	
\begin{figure}[!h]	
	\includegraphics[width=1\linewidth,]{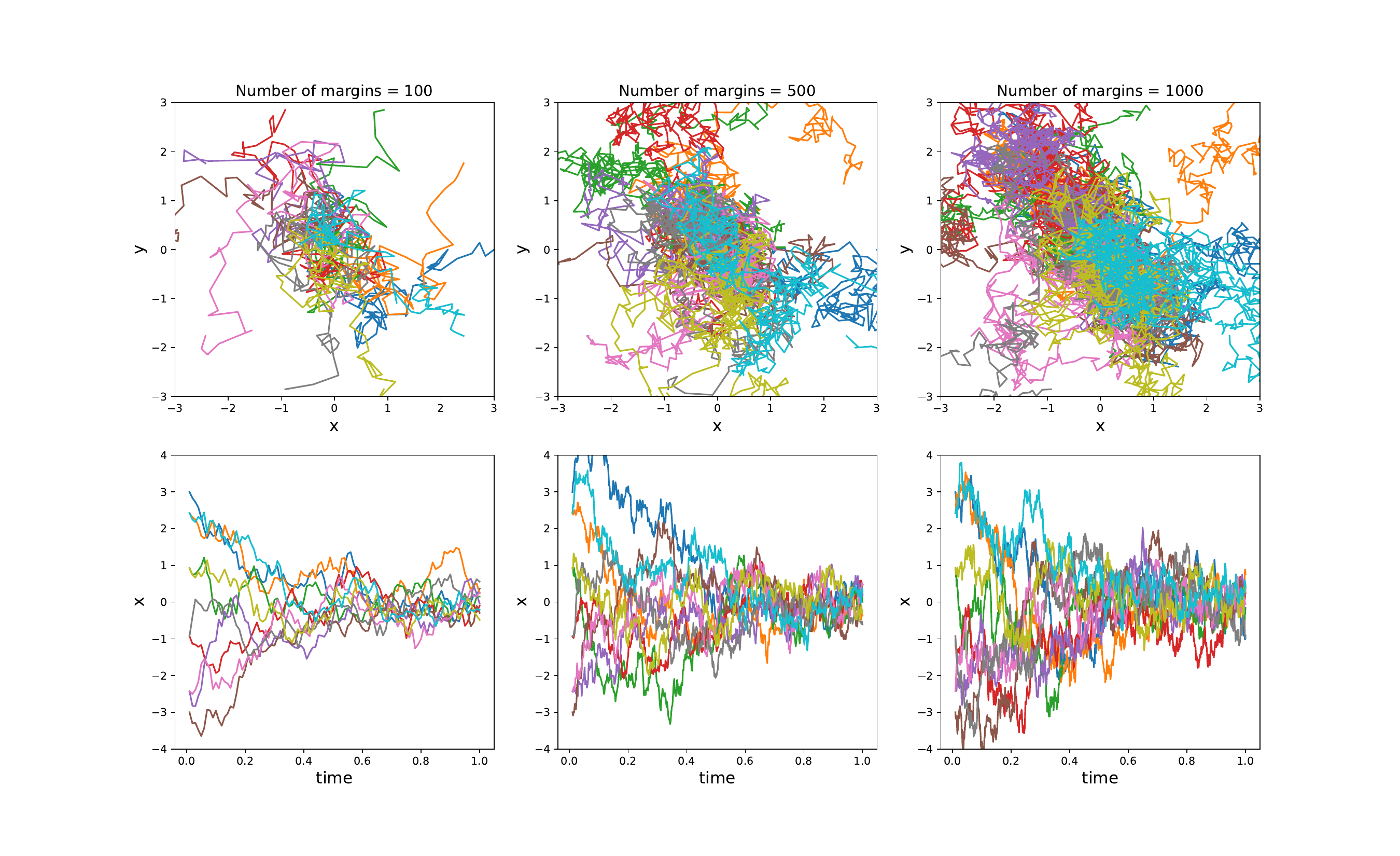}
	\caption{Sample paths reconstructed with entropic optimal transport with {\bf fractional Brownian motion} reference. The Hurst index has been set to $H=0.25$. Top panels: two dimensional projection of the full trajectories. Bottom panels: evolution over time of the $x$-axis.}
	\label{fig:fbm_ref_eot_paths}

\end{figure}

\section{Discussion}

We conclude with a qualitative discussion of the two components composing the objective function, namely the optimal transport term and the Kullback--Leibler term, and their respective roles in shaping regularity. Consider first the unregularized pairwise optimal transport problem
\begin{equation*}
\min_{\pi\in\Pi(\mu_{t_j},\mu_{t_{j+1}})} \int_{\mathbb{R}^d\times \mathbb{R}^d} \|x-y\|^2 \pi(dxdy).
\end{equation*}
Among all admissible couplings, this problem selects the tightest one, hence the most strongly correlated, and therefore induces the smoothest possible interpolation between $\mu_{t_j}$ and $\mu_{t_{j+1}}$. When such couplings are composed across time, classical Kolmogorov--\v{C}entsov-type arguments imply that strong short-time correlations translate into regular sample paths. In this sense, pure optimal transport acts as a \emph{maximum smoothness principle}. This principle is well suited when observations are dense in time and accurately measured. However, when time points are sparse 
it can become overly rigid: optimal transport then favors nearly deterministic transitions over long intervals, suppressing variability at large scales. 
Introducing a reference measure through an entropic penalty provides a way to relax this rigidity by enforcing a form of \emph{controlled roughness}. Augmenting the objective with
\begin{equation*}
	\mathrm{KL}(\pi|\pi_{\refe})
\end{equation*}
does more than stabilize computation: it introduces a competing structural bias. While the transport term promotes maximal correlation and smoothness, the Kullback--Leibler term penalizes departures from the reference coupling. This prevents the inferred dynamics from becoming smoother or more strongly correlated than what the reference deems plausible. When the reference coupling is induced by a continuous-time Gaussian process, this trade-off becomes inherently scale-dependent: the reference encodes how correlation should decay with time, so that short time gaps favor smooth transitions, while longer gaps allow for increased variability. This is particularly advantageous when observation times are irregular or sparse. From this perspective, the limitations of product reference couplings become apparent. A product reference corresponds to maximal roughness, enforcing complete decorrelation between successive states regardless of the temporal spacing. Such references are therefore not only dynamically trivial but also misaligned with the notion of temporal coherence. By contrast, Gaussian reference processes encode temporal regularity in a structured and tunable way. For example, Mat\'ern-type processes allow one to directly control the regularity of the inferred dynamics through a smoothness parameter, interpolating between rough, noise-dominated behavior and highly regular trajectories. In this sense, the entropic penalty acts as a scale-aware and tunable roughness prior.

In closing this section, we remark that our reconstruction is related to Schr\"odinger bridge problems \cite{leonardsurveySchrodinger2013, bernton2019schr, pavon2021data, hong2025trajectory}, which induce stochastic dynamics between prescribed endpoint distributions relative to a reference process. However, classical Schr\"odinger bridges presuppose a \emph{global} reference dynamics over the entire time interval. This is a modelling choice that may or may not be suitable, depending on the context. It is conceivable that it sometimes would be difficult to justify a ``global prior" statistically when only marginal information is available. In such cases, the locality of our framework is well-suited: the reference process is used only to ensure consistent \emph{pairwise} transitions
, but does not bias to the global behavior of the reference: rather than postulating a full global dynamics a priori, we let local Gaussian reference couplings act as modular building blocks, from which more complex dynamics can be assembled or iteratively refined. In this sense, our approach can be viewed as a statistically conservative counterpart to Schr\"odinger bridges, lying between static entropic optimal transport and full Schrödinger bridge formulations. Importantly, this locally decomposable formulation admits an explicit closed-form characterization in the Gaussian case. 

\color{black}

\begin{figure}[h]

	\includegraphics[width=0.9\linewidth,]{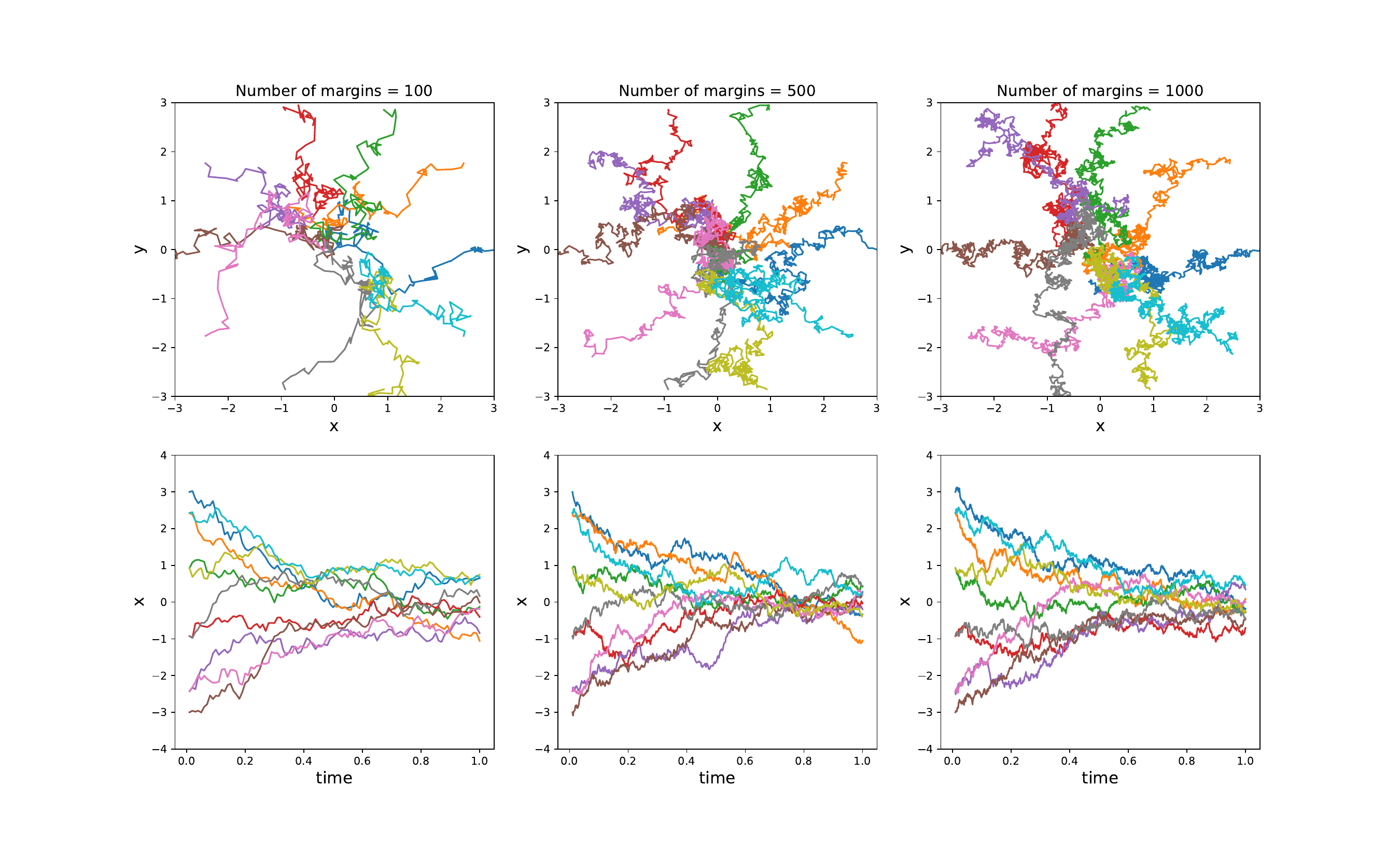}
	\caption{Sample paths reconstructed with entropic optimal transport and {\bf Brownian motion} as reference. Top panels: two dimensional projection of the full trajectories. Bottom panels: evolution over time of the $x$-axis.}
	\label{fig:brown_ref_eot_paths}
\end{figure}	
	\begin{figure}[h]	
	\includegraphics[width=1\linewidth,]{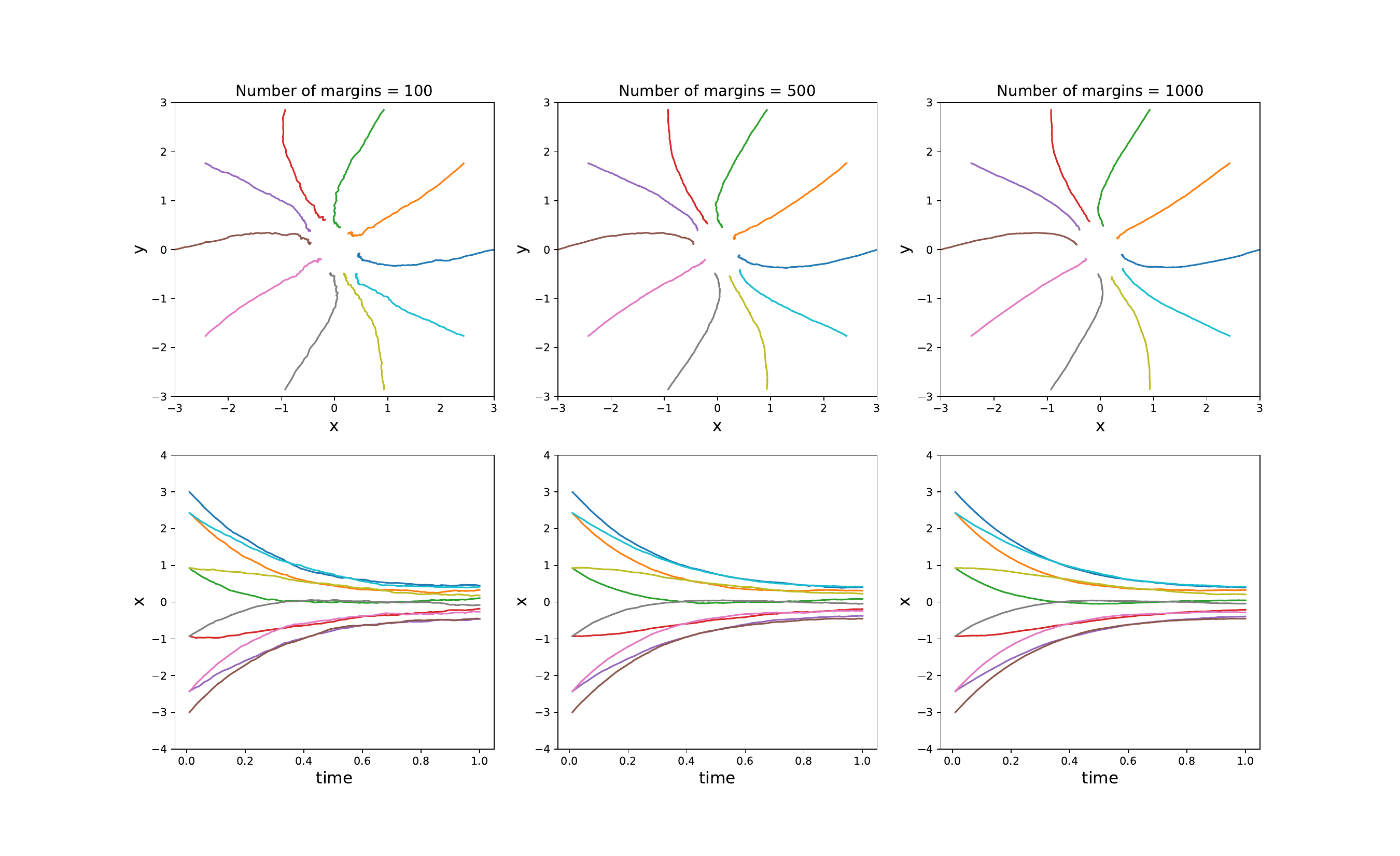}
	\caption{Sample paths reconstructed with entropic optimal transport and {\bf heat kernel} as reference. Top panels: two dimensional projection of the full trajectories. Bottom panels: evolution over time of the $x$-axis.}
	\label{fig:heat_ref_eot_paths}
\end{figure}

\newpage

\color{black}
\bibliography{Work_SMAT}

\begin{thebibliography}{10}

\bibitem{ambrosioGradientflowsmetric2005a}
L.~Ambrosio, N.~Gigli, and G.~Savar{\'e}.
\newblock {\em Gradient Flows: In Metric Spaces and in the Space of Probability
  Measures}.
\newblock Lectures in Mathematics {{ETH Z\"urich}}. {Birkh\"auser}, {Boston},
  2005.

\bibitem{baker1973joint}
C.~R. Baker.
\newblock Joint measures and cross-covariance operators.
\newblock {\em Transactions of the American Mathematical Society},
  186:273--289, 1973.

\bibitem{bakonyi2011matrix}
M.~Bakonyi and H.~J. Woerdeman.
\newblock {\em Matrix completions, moments, and sums of Hermitian squares}.
\newblock Princeton University Press, 2011.

\bibitem{bengtsson2017geometry}
I.~Bengtsson and K.~{\.Z}yczkowski.
\newblock {\em Geometry of quantum states: an introduction to quantum
  entanglement}.
\newblock Cambridge university press, 2017.

\bibitem{bernton2019schr}
E.~Bernton, J.~Heng, A.~Doucet, and P.~E. Jacob.
\newblock Schr$\backslash$" odinger bridge samplers.
\newblock {\em arXiv preprint arXiv:1912.13170}, 2019.

\bibitem{bhatia2009positive}
R.~Bhatia.
\newblock Positive definite matrices.
\newblock In {\em Positive Definite Matrices}. Princeton university press,
  2009.

\bibitem{bhatia2019bures}
R.~Bhatia, T.~Jain, and Y.~Lim.
\newblock On the bures-wasserstein distance between positive definite matrices.
\newblock {\em Expositiones Mathematicae}, 37(2):165--191, 2019.

\bibitem{boyd2004convex}
S.~P. Boyd and L.~Vandenberghe.
\newblock {\em Convex optimization}.
\newblock Cambridge university press, 2004.

\bibitem{brenier1991polar}
Y.~Brenier.
\newblock Polar factorization and monotone rearrangement of vector-valued
  functions.
\newblock {\em Communications on pure and applied mathematics}, 44(4):375--417,
  1991.

\bibitem{chewi2024statistical}
S.~Chewi, J.~Niles-Weed, and P.~Rigollet.
\newblock Statistical optimal transport.
\newblock {\em arXiv preprint arXiv:2407.18163}, 3, 2024.

\bibitem{chizat2020faster}
L.~Chizat, P.~Roussillon, F.~L{\'e}ger, F.-X. Vialard, and G.~Peyr{\'e}.
\newblock Faster wasserstein distance estimation with the sinkhorn divergence.
\newblock {\em Advances in Neural Information Processing Systems},
  33:2257--2269, 2020.

\bibitem{cuesta1996lower}
J.~A. Cuesta-Albertos, C.~Matr{\'a}n-Bea, and A.~Tuero-Diaz.
\newblock On lower bounds for the l 2-wasserstein metric in a hilbert space.
\newblock {\em Journal of Theoretical Probability}, 9(2):263--283, 1996.

\bibitem{cuturi2013sinkhorn}
M.~Cuturi.
\newblock Sinkhorn distances: Lightspeed computation of optimal transport.
\newblock {\em Advances in neural information processing systems}, 26, 2013.

\bibitem{del2020statistical}
E.~del Barrio and J.-M. Loubes.
\newblock The statistical effect of entropic regularization in optimal
  transportation.
\newblock {\em arXiv preprint arXiv:2006.05199}, 2020.

\bibitem{fischer2019inferring}
D.~S. Fischer, A.~K. Fiedler, E.~M. Kernfeld, R.~M. Genga, A.~Bastidas-Ponce,
  M.~Bakhti, H.~Lickert, J.~Hasenauer, R.~Maehr, and F.~J. Theis.
\newblock Inferring population dynamics from single-cell rna-sequencing time
  series data.
\newblock {\em Nature biotechnology}, 37(4):461--468, 2019.

\bibitem{givens1984class}
C.~R. Givens and R.~M. Shortt.
\newblock A class of wasserstein metrics for probability distributions.
\newblock {\em Michigan Mathematical Journal}, 31(2):231--240, 1984.

\bibitem{hong2025trajectory}
W.~Hong, Y.~Shi, and J.~Niles-Weed.
\newblock Trajectory inference with smooth schr$\backslash$" odinger bridges.
\newblock {\em arXiv preprint arXiv:2503.00530}, 2025.

\bibitem{janati2020entropic}
H.~Janati, B.~Muzellec, G.~Peyr{\'e}, and M.~Cuturi.
\newblock Entropic optimal transport between unbalanced gaussian measures has a
  closed form.
\newblock {\em Advances in neural information processing systems},
  33:10468--10479, 2020.

\bibitem{leonardsurveySchrodinger2013}
C.~L{\'e}onard.
\newblock A survey of the schrodinger problem and some of its connections with
  optimal transport, 2013.

\bibitem{magnus2019matrix}
J.~R. Magnus and H.~Neudecker.
\newblock {\em Matrix differential calculus with applications in statistics and
  econometrics}.
\newblock John Wiley \& Sons, 2019.

\bibitem{mallastoEntropyregularized2WassersteinDistance2022}
A.~Mallasto, A.~Gerolin, and H.~Q. Minh.
\newblock Entropy-regularized 2-{{Wasserstein}} distance between {{Gaussian}}
  measures.
\newblock {\em Information Geometry}, 5(1):289--323, 2022.

\bibitem{marino2020optimal}
S.~D. Marino and A.~Gerolin.
\newblock An optimal transport approach for the schr{\"o}dinger bridge problem
  and convergence of sinkhorn algorithm.
\newblock {\em Journal of Scientific Computing}, 85(2):27, 2020.

\bibitem{minh2023}
H.~Q. Minh.
\newblock Entropic regularization of wasserstein distance between
  infinite-dimensional gaussian measures and gaussian processes.
\newblock {\em Journal of Theoretical Probability}, 36(1):201--296, 2023.

\bibitem{nielsen2001quantum}
M.~A. Nielsen and I.~L. Chuang.
\newblock {\em Quantum computation and quantum information}, volume~2.
\newblock Cambridge university press Cambridge, 2001.

\bibitem{nutz2021introduction}
M.~Nutz.
\newblock Introduction to entropic optimal transport.
\newblock {\em Lecture notes, Columbia University}, 2021.

\bibitem{panaretos2020invitation}
V.~M. Panaretos and Y.~Zemel.
\newblock {\em An invitation to statistics in Wasserstein space}.
\newblock Springer Nature, 2020.

\bibitem{pardo2018statistical}
L.~Pardo.
\newblock {\em Statistical inference based on divergence measures}.
\newblock Chapman and Hall/CRC, 2018.

\bibitem{pavliotis2014stochastic}
G.~A. Pavliotis.
\newblock Stochastic processes and applications.
\newblock {\em Texts in applied mathematics}, 60, 2014.

\bibitem{pavon2021data}
M.~Pavon, G.~Trigila, and E.~G. Tabak.
\newblock The data-driven schr{\"o}dinger bridge.
\newblock {\em Communications on Pure and Applied Mathematics},
  74(7):1545--1573, 2021.

\bibitem{peyre2019computational}
G.~Peyr{\'e}, M.~Cuturi, et~al.
\newblock Computational optimal transport: With applications to data science.
\newblock {\em Foundations and Trends{\textregistered} in Machine Learning},
  11(5-6):355--607, 2019.

\bibitem{rigollet2025sample}
P.~Rigollet and A.~J. Stromme.
\newblock On the sample complexity of entropic optimal transport.
\newblock {\em The Annals of Statistics}, 53(1):61--90, 2025.

\bibitem{takatsu2011wasserstein}
A.~Takatsu.
\newblock Wasserstein geometry of gaussian measures.
\newblock {\em Osaka J. Math.}, 2011.

\bibitem{villani2009optimal}
C.~Villani.
\newblock {\em Optimal transport: old and new}, volume 338.
\newblock Springer, 2009.

\bibitem{villani2021topics}
C.~Villani.
\newblock {\em Topics in optimal transportation}, volume~58.
\newblock American Mathematical Soc., 2021.

\bibitem{yun2025spectral}
H.~Yun.
\newblock Spectral shinkage of gaussian entropic optimal transport.
\newblock {\em arXiv preprint arXiv:2512.19457}, 2025.

\end{thebibliography}

\newpage

\appendix

\section{Proofs related to the dual problem approach}\label{sec:proofs_dual}

\paragraph{Proof of Proposition \ref{prop:duality_OTGauss}}\label{proof:dual_problem}

\begin{proof}
	We start from the primal problem associated to $W_{\Sigma}^\varepsilon(\mu,\nu)$:
	\begin{equation*}
		\min_{X \in S_{2d}^{++}(\mathbb{R})}\langle Y + \varepsilon  \Sigma^{-1}, X \rangle_{\rm HS} -  \varepsilon \log \det X \quad \text{such that} \quad X_{11} = A,~ X_{22} = B.
	\end{equation*}
	To substitute the constraints $X_{11} = A$ and $X_{22} = B$ that appear in last expression, we introduce the function $g : S_{2d}(\mathbb{R}) \rightarrow \mathbb{R} \cup \{+\infty\}$ defined for every $X \in S_{2d}(\mathbb{R})$ by  
	\begin{equation}
		g(X) := \sup_{(P,Q) \in S_d \times S_d} \langle X_{11} - A, P \rangle_{\HS}  + \langle X_{22} - B, Q \rangle_{\HS}.
	\end{equation}
	
	Also, to work with a convex function on the vector space $S_{2d}(\mathbb{R})$ instead of the cone $S_{2d}^{++}(\mathbb{R})$, we set the log-determinant function to be $+\infty$ outside $S_{2d}^{++}(\mathbb{R})$. At $X \in S_{2d}(\mathbb{R})$, this extended log-determinant function\footnote{Extending the log-determinant by setting $-\log\det(X)=+\infty$ only if $\det(X) \leq 0$ would not yield a convex function on $S_{2d}(\mathbb{R})$.}, that we denote by $\varphi$, takes value
	\begin{equation*}
		\varphi(X) = \begin{cases}
			-\log \det(X), & \text{if}~ X \in S_{2d}^{++}(\mathbb{R})\\        + \infty, &\text{otherwise}.
		\end{cases}
	\end{equation*}
	To finish rewriting problem \eqref{eq:dual_start_comput}, we denote by $f$ the function $f : S_{2d}(\mathbb{R}) \rightarrow \mathbb{R} \cup \{+\infty\}$ defined by 
	\begin{equation}
		f(X) = \langle Y + \varepsilon \Sigma^{-1},X \rangle_{\HS} + \varepsilon \varphi(X).
	\end{equation}
	With these notations, our primal optimization problem \eqref{eq:dual_start_comput} reads
	\begin{equation}
		\min_{X \in S_{2d}(\mathbb{R})} f(X) + g(X).
	\end{equation}
	Applying Fenchel-Legendre duality Theorem \ref{thm:duality_convex}, we derive the equality
	\begin{equation}\label{eq:reminder_duality}
		\min_{X \in S_{2d}(\mathbb{R})} f(X) + g(X) = \max_{Z \in S_{2d}} -f^{*}(-Z) - g^*(Z),
	\end{equation}
	where $f^*$ and $g^*$ are the Legendre transform of $f$ and $g$ respectively. To compute $f^*$, we set $Z \in S_{2d}(\mathbb{R})$ and write
	\begin{align*}
		f^*(Z) &= \sup_{X \in S_{2d}(\mathbb{R})} {\langle Z, X \rangle_{\HS} - f(X)}\\
		& = \sup_{X \in S_{2d}^{++}(\mathbb{R})} {\langle Z-Y- \varepsilon \Sigma^{-1}, X \rangle_{\HS} - \varepsilon \varphi(X)} \\
		& = \varepsilon \sup_{X \in S_{2d}^{++}(\mathbb{R})} {\left\langle \frac{Z-Y}{\varepsilon} - \Sigma^{-1}, X \right \rangle_{\HS} -  \varphi(X)} \\
		& = \varepsilon \varphi^*\left(\frac{Z-Y}{\varepsilon} - \Sigma^{-1} \right) \\
		& = - \varepsilon \log \det \left( \Sigma^{-1} +\frac{Y-Z}{\varepsilon}\right) - 2\varepsilon d,
	\end{align*}
	using that the Legendre transform of the negative log-determinant $\varphi$ is well-known to derive last equality. Indeed, $\varphi^*$ is computed for example in \citep[p.~92,~Ex.~3.23]{boyd2004convex}, and given for every $V \in S_{2d}^{++}(\mathbb{R})$ by the formula 
	\begin{equation}
		\varphi^*(V) = -\log\det(-V) -2d.
	\end{equation}
	In the expression of $f^*$, in the case, $Y-Z$ does not belong to $S_{2d}^{++}(\mathbb{R})$, the value $-\log \det(Y-Z)$ is equal to $+\infty$.
	To compute $g^*$, we write $Z = \begin{pmatrix}
		F & K \\
		K^T&G
	\end{pmatrix}$ and
	\begin{align}
		g^*(Z)& = \sup_{X \in S_{2d}} \langle Z, X \rangle_{\HS} - \sup_{(P,Q) \in S_d \times S_d}\left\langle \begin{pmatrix}
			A & 0 \\
			0 & B
		\end{pmatrix} -\begin{pmatrix}
			X_{11} & X_{12} \\
			X_{12}^T& X_{22}
		\end{pmatrix}, \begin{pmatrix}
			P & 0 \\
			0 & Q
		\end{pmatrix} \right\rangle_{\HS} \\
		& = \sup_{X_{12} \in M_{d}}  \left\langle \begin{pmatrix}
			F & K \\
			K^T&G
		\end{pmatrix},  \begin{pmatrix}
			A & X_{12} \\
			X_{12}^*&  B
		\end{pmatrix} \right\rangle_{\HS} \\
		& = \begin{cases}
			\langle F, A \rangle_{\HS} + \langle G, B \rangle_{\HS} & \text{if}~ K = 0 \\
			+\infty & \text{otherwise.}
		\end{cases}
	\end{align}
	Maximizing the right hand side of \eqref{eq:reminder_duality} constraints to maximize over the matrices of the form $Z = \begin{pmatrix}
		F & 0 \\
		0 & G
	\end{pmatrix}$ with $(F,G) \in S_d \times S_d$.  After these computations, we have that 
	\begin{align*}
		-f^*(-Z) - g^*(Z) &=  \varepsilon \log \det \left( \Sigma^{-1} +\frac{Y+Z}{\varepsilon}\right) + 2\varepsilon d - \langle F, A \rangle_{\HS} - \langle G, B \rangle_{\HS} \\
		& = - \langle F, A \rangle_{\HS} - \langle G, B \rangle_{\HS} +  \varepsilon \log \det \begin{pmatrix}
			\Gamma_{11} +   \varepsilon^{-1}(I_d + F) & \Gamma_{12}-\varepsilon^{-1}I_d \\
			\Gamma_{12}^T-\varepsilon^{-1}I_d &  \Gamma_{22} +\varepsilon^{-1}(I_d + G)
		\end{pmatrix} + 2\varepsilon d\\
		& = -\langle F, A \rangle_{\HS} - \langle G, B \rangle_{\HS}  +  \varepsilon \log \left( \varepsilon^{-2d} \det \begin{pmatrix}
			(I_d + F) + \varepsilon \Gamma_{11}   & -I_d + \varepsilon \Gamma_{12}     \\
			- I_d  +   \varepsilon \Gamma_{12}^T&   (I_d + G) +  \varepsilon \Gamma_{22}
		\end{pmatrix} \right) + 2\varepsilon d \\ 
		& = -\langle F, A \rangle_{\HS} - \langle G, B \rangle_{\HS}  +  \varepsilon \log  \det \begin{pmatrix}
			(I_d + F) + \varepsilon \Gamma_{11}   & -I_d + \varepsilon \Gamma_{12}     \\
			- I_d  +   \varepsilon \Gamma_{12}^T&   (I_d + G) +  \varepsilon \Gamma_{22}
		\end{pmatrix} + 2\varepsilon d(1-\log(\varepsilon)). \\ 
	\end{align*}
	Ignoring for the moment the additive constant $2\varepsilon d(1-\log(\varepsilon))$, and making use of the notation $M_{\varepsilon}= I_d- \varepsilon \Gamma_{12}$, the right hand side of equation \eqref{eq:reminder_duality} reads
	\begin{equation}
		\max_{(F,G) \in S_d \times S_d}  \langle -F, A \rangle_{\HS} + \langle -G, B \rangle_{\HS} +   \varepsilon \log  \det \begin{pmatrix}
			(I_d + F) + \varepsilon \Gamma_{11}   & -M_{\varepsilon}    \\
			- M_{\varepsilon}^T&   (I_d + G) +  \varepsilon \Gamma_{22}
		\end{pmatrix}.
	\end{equation}

	After the changes of variable $F = I_d+F+\varepsilon \Gamma_{11}$ and $G = I_d+G+\varepsilon \Gamma_{22}$, which are licit as if $F$ and $G$ belong to $S_d$; so do $I_d+F+\varepsilon \Gamma_{11}$ and $I_d + G + \varepsilon \Gamma_{22}$, we derive 
	\begin{equation}
		W_\Sigma^\varepsilon(\mu, \nu) = \max_{(F,G) \in S_d \times S_d}  \langle I_d + \varepsilon \Gamma_{11} - F, A \rangle_{\HS} + \langle I_d + \varepsilon \Gamma_{22} -G, B \rangle_{\HS} +   \varepsilon \log  \det \begin{pmatrix}
			F   & -M_{\varepsilon}   \\
			-M_{\varepsilon}^T&    G 
		\end{pmatrix}.
	\end{equation}
	
	To conclude, recall that the function $-\log \det\begin{pmatrix}
		F   & -M_{\varepsilon}    \\
		-M_{\varepsilon}^T&    G 
	\end{pmatrix}$ equals $+\infty$ if the matrix $\begin{pmatrix}
		F   & -M_{\varepsilon}    \\
		-M_{\varepsilon}^T&    G 
	\end{pmatrix}$ is not positive-definite. A necessary condition for this to hold is that $F$ and $G$ are positive definite. We can thus reduce the constraint space to $S_d^{++}(\mathbb{R}) \times S_d^{++}(\mathbb{R})$; instead of $S_d(\mathbb{R}) \times S_d(\mathbb{R})$.
\end{proof}

\paragraph{Proof of Proposition \ref{prop:grad_dual}}

\begin{proof}
	The dual function is a sum of a linear term and a log-determinant term. The gradient of the linear term is constant and equal to $(-A,-B)$. Regarding the log-determinant term, if $(F,G)$ is such that $G-M_\varepsilon F^{-1} M_\varepsilon^{T}$ is positive-definite, from Theorem \ref{thm:positive_mat_positive_schur} the matrix 
	\begin{equation}
		\begin{pmatrix}
			F & -M_{\varepsilon} \\
			-M_{\varepsilon} & G
		\end{pmatrix}
	\end{equation}
	is positive-definite. We now exploit that the log-determinant is differentiable on $S_{2d}^{++}(\mathbb{R})$ and that its gradient is given by the inverse matrix. Moreover, in our case, only first variations of the form $H = \diag(H_1,H_2)$ are allowed. More precisely, introducing the function $\widetilde{D}_{\Sigma}^{\varepsilon}$ defined at $F,G$ by 
	\begin{equation}
		\widetilde{D}_{\Sigma}^{\varepsilon}(F,G) := \log \det  \begin{pmatrix}
			F & -M_{\varepsilon}\\
			-M_{\varepsilon}^T& G
		\end{pmatrix},
	\end{equation}
	we write
	\begin{align*}
		\widetilde{D}_{\Sigma}^{\varepsilon}(F + H_1,G + H_2) & =  \log \det \left(   \begin{pmatrix}
			F & -M_{\varepsilon}\\
			-M_{\varepsilon}^T& G
		\end{pmatrix} +  \begin{pmatrix}
			H_1 & 0\\
			0 & H_2
		\end{pmatrix} \right)\\
		& = \log \det  \begin{pmatrix}
			F & -M_{\varepsilon}\\
			-M_{\varepsilon}^T& G
		\end{pmatrix} + \left \langle \begin{pmatrix}
			F & -M_{\varepsilon}\\
			-M_{\varepsilon}^T& G
		\end{pmatrix}^{-1} ,   \begin{pmatrix}
			H_1 & 0\\
			0 & H_2
		\end{pmatrix} \right \rangle_{\rm HS} + o(H),
	\end{align*}
	where we used that that the gradient of the log-det function at point $X$ is its inverse $X^{-1}$ (see e.g. \cite[p.~641]{boyd2004convex}). Now, exploiting the inversion formula for block matrices, we derive that
	\begin{align*}
		\left \langle \begin{pmatrix}
			F & -M_{\varepsilon}\\
			- M_{\varepsilon}^T& G
		\end{pmatrix}^{-1} , \begin{pmatrix}
			H_1 & 0\\
			0 & H_2
		\end{pmatrix} \right \rangle_{\rm HS}& = \left \langle \begin{pmatrix}
			F^{-1}+F^{-1}M_\varepsilon S^{-1}M_\varepsilon^T F^{-1} & (-) \\
			(-) & S^{-1}
		\end{pmatrix} , \begin{pmatrix}
			H_1 & 0\\
			0 & H_2
		\end{pmatrix} \right \rangle_{\rm HS}\\ 
		&= \langle F^{-1} + F^{-1} M_{\varepsilon}S^{-1} M_{\varepsilon}^TF^{-1}, H_1 \rangle_{\HS} + \langle  S^{-1}, H_2 \rangle_{\HS},
	\end{align*}
	where $S = G- M_\varepsilon^TF^{-1}M_\varepsilon$. Collecting the pieces together, we get 
	\begin{equation*}
		\nabla D_{\Sigma}^\varepsilon(F,G) = (\varepsilon (F^{-1}+ F^{-1} M_{\varepsilon}S^{-1} M_{\varepsilon}^TF^{-1}) -A,   \varepsilon S^{-1}-B),
	\end{equation*}
	as claimed. Now, the first order optimality condition, that is the equation $\nabla D_\Sigma^\varepsilon(F,G) = 0$ is equivalent to the system
	\begin{equation}\label{eq:first_dual}
		\left\{ 
		\begin{array}{ccc}
			\varepsilon(F^{-1}+F^{-1} M_{\varepsilon}S^{-1} M_{\varepsilon}^TF^{-1}  ) & = & A \\
			\varepsilon S^{-1} & = & B.
		\end{array}
		\right.
	\end{equation}
	Exploiting the second equation, we can substitute $S^{-1}$ by $\varepsilon^{-1}B$ in the first equation to rewrite the first equation 
	\begin{equation}
		\varepsilon F^{-1}+ F^{-1} M_{\varepsilon}B M_{\varepsilon}^TF^{-1}    =  A.
	\end{equation}
	Then, we have the equivalences
	\begin{align*}
		\varepsilon F^{-1}+ F^{-1} M_{\varepsilon}B M_{\varepsilon}^TF^{-1}  =  A & \Leftrightarrow \varepsilon F + M_{\varepsilon}BM_{\varepsilon}^T  = FAF \\
		& \Leftrightarrow  FAF - \varepsilon F  -M_{\varepsilon}BM_{\varepsilon}^T= 0.\\
	\end{align*}
	For the second equation, we derive 
	\begin{align*}
		\varepsilon S^{-1}  =  B & \Leftrightarrow \varepsilon B^{-1} = S \\
		& \Leftrightarrow  \varepsilon B^{-1} = G -M^TF^{-1}M\\
		& \Leftrightarrow   G =  \varepsilon B^{-1} + M^TF^{-1}M.
	\end{align*}
	We thus have shown that the matrix equations system \eqref{eq:first_dual} is equivalent to the system
	\begin{equation*}
		\left\{ 
		\begin{array}{ccc}
			FAF - \varepsilon F - M_{\varepsilon}BM_{\varepsilon}^T& = & 0 \\
			\varepsilon B^{-1} + M_\varepsilon^T F^{-1}M_{\varepsilon}& = & G.
		\end{array}
		\right.
	\end{equation*}
\end{proof}

\paragraph{Proof of Theorem \ref{thm:optim_potentials}}

\begin{proof}
	Our strategy is to solve the system \eqref{eq:system_dual} starting from the first matrix equation 
	\begin{align*}
		FAF - \varepsilon F - M_{\varepsilon}BM_{\varepsilon}^T= 0 & \Leftrightarrow  AFAF - \varepsilon AF - AM_{\varepsilon}BM_{\varepsilon}^T= 0 \\
		& \Leftrightarrow Z^2 - \varepsilon Z - AM_{\varepsilon}BM_{\varepsilon}^T= 0, 
	\end{align*}
	with the notation $Z = AF$. This last matrix equation is very similar to \eqref{eq:mat_eq_primal}, that we solved for proving Theorem \ref{thm:closed_form_gauss}. Adapting the argument, we establish that the equation ${Z^2 - \varepsilon Z - AM_\varepsilon B M_\varepsilon^T}$ has a unique solution $Z_\varepsilon$ such that $F_\varepsilon := A^{-1}Z_\varepsilon$ is positive-definite. This solution is given by the matrix $Z_\varepsilon$ defined by the formula
	\begin{equation*}
		Z_\varepsilon :=\frac{\varepsilon}{2}I_d+ \left(AM_{\varepsilon}BM_{\varepsilon}^T+ \frac{\varepsilon^2}{4}I_d  \right)^{1/2}.
	\end{equation*}
	Then, as $Z_{\varepsilon} = AF_\varepsilon$, we have $F_\varepsilon = A^{-1} Z_{\varepsilon}$. This yields 
	\begin{equation}\label{eq:inproof_solution_K}
		F_\varepsilon =  A^{-1}\left( \frac{\varepsilon}{2}I_d +\left(AM_{\varepsilon}BM_{\varepsilon}^T+ \frac{\varepsilon^2}{4}I_d  \right)^{1/2} \right).
	\end{equation}
	Let us now compute the second dual variable. To do so, let us go back to system \eqref{eq:system_dual}; that we repeat below for clarity:
	\begin{equation*}
		\left\{ 
		\begin{array}{ccc}
			FAF - \varepsilon F - M_{\varepsilon}BM_{\varepsilon}^T& = & 0 \\
			\varepsilon B^{-1} + M_\varepsilon^TF^{-1}M_{\varepsilon}& = & G.
		\end{array}
		\right.
	\end{equation*}
	Notice that we can rewrite the first equation as follows:
	\begin{align*}
		FAF - \varepsilon F - M_{\varepsilon}BM_{\varepsilon}^T= 0 & \Leftrightarrow  FA - \varepsilon I_d - M_{\varepsilon}BM_\varepsilon^TF^{-1} = 0 \\
		& \Leftrightarrow  M_\varepsilon^T F^{-1} = B^{-1}M_{\varepsilon}^{-1}(FA - \varepsilon I_d) \\
		& \Leftrightarrow   M_\varepsilon^TF^{-1}M_{\varepsilon}= B^{-1}M_\varepsilon^{-1}(FA - \varepsilon I_d)M_{\varepsilon}.
	\end{align*}
	This expression of the matrix $M_{\varepsilon}^TF^{-1}M_{\varepsilon}$ allows us to rewrite the second equation of the system \eqref{eq:system_dual} as
	\begin{align*}
		G &= B^{-1}( \varepsilon I_d + M_{\varepsilon}^{-1}(FA - \varepsilon I_d)M_{\varepsilon})\\
		& = B^{-1}M_{\varepsilon}^{-1}FAM_{\varepsilon}.
	\end{align*}
	Introducing $F_{\varepsilon}$ solution of the first equation, we express it as in equation \eqref{eq:inproof_solution_K} to write $G_{\varepsilon}$ solution of the system as 
	\begin{equation}
		G_\varepsilon  = B^{-1} \left( M_{\varepsilon}^{-1}A^{-1}\left( \frac{\varepsilon}{2}I_d + \left(AM_{\varepsilon}BM_{\varepsilon}^T+ \frac{\varepsilon^2}{4}I_d  \right)^{1/2} \right)AM_{\varepsilon}\right).
	\end{equation}
	Now, the identity
	\begin{equation*}
		M_\varepsilon^{-1}A^{-1}\left(AM_{\varepsilon}BM_{\varepsilon}^T+ \frac{\varepsilon^2}{4}I_d  \right)^{1/2} AM_{\varepsilon}=\left(BM_{\varepsilon}^TAM_{\varepsilon}+ \frac{\varepsilon^2}{4}I_d  \right)^{1/2}
	\end{equation*}
	yields 
	\begin{equation}
		G_{\varepsilon} = B^{-1}\left( \frac{\varepsilon}{2}I_d +\left(BM_{\varepsilon}^TAM_{\varepsilon}+ \frac{\varepsilon^2}{4}I_d  \right)^{1/2} \right).
	\end{equation}
	If the dual function \eqref{eq:dual_objective} reaches a maximum, it is on the subset $\Pi(M_\varepsilon)$ introduced in equation \eqref{eq:constraint_dual}. Exploiting Theorem \ref{thm:positive_mat_positive_schur}, one can show that the dual matrices $F_\varepsilon$ and $G_\varepsilon$ are such that the matrix 
	\begin{equation}
		\begin{pmatrix}
			F_\varepsilon & -M_\varepsilon \\
			-M_\varepsilon^T & G_\varepsilon
		\end{pmatrix}
	\end{equation}
	is positive-definite. This shows that $(F_\varepsilon, G_\varepsilon)$ belongs to $\Pi(M_\varepsilon)$. Moreover, as the dual objective function is strictly concave on the set $\Pi(M_\varepsilon)$, and $\nabla D_{\Sigma}^\varepsilon(F_\varepsilon, G_\varepsilon) = (0,0)$, we deduce that the couple $(F_\varepsilon, G_\varepsilon)$ is the unique solution to the dual problem \eqref{eq:duality_OTGauss}.
\end{proof}

\paragraph{Proof of Corollary \ref{cor:eot_cost_dual}}

\begin{proof}
	Using the notations from Theorem \ref{thm:optim_potentials}, we can write the optimal transport cost between $\mu$ and $\nu$ as
	\begin{equation*}
		W_\Sigma^{\varepsilon}(\mu, \nu) = \langle I_d+\varepsilon \Gamma_{11} - F_{\varepsilon}^{\star}, A \rangle_{\HS} + \langle I_d + \varepsilon \Gamma_{22} - G_\varepsilon^{\star}, B \rangle_{\HS} + \varepsilon \log \det \begin{pmatrix}
			F_\varepsilon^{\star}  & -M_{\varepsilon}    \\
			-M_{\varepsilon}^T&    G_\varepsilon^{\star}
		\end{pmatrix} - \varepsilon \log \det(\varepsilon \Sigma^{-1}).
	\end{equation*}
	For the scalar product terms, we begin by writing 
	\begin{align*}
		\langle I_d +\varepsilon \Gamma_{11} - F_{\varepsilon}^{\star}, A \rangle_{\HS} & = \tr(A) + \varepsilon \langle \Gamma_{11}, A \rangle_{\HS} - \tr\left(\frac{\varepsilon}{2}I_d+ \left(AM_{\varepsilon}BM_{\varepsilon}^T+ \frac{\varepsilon^2}{4}I_d \right)^{1/2} \right) \\
		&  =  \tr(A) + \varepsilon \langle \Gamma_{11}, A \rangle_{\HS} - \tr\left( \left(AM_{\varepsilon}BM_{\varepsilon}^T+ \frac{\varepsilon^2}{4}I_d \right)^{1/2} \right) - \varepsilon \frac{d}{2}.
	\end{align*}
	And second we write, 
	\begin{align*}
		\langle I_d+\varepsilon \Gamma_{22}- G_{\varepsilon}^{\star}, B \rangle_{\HS} & = \tr(B) + \varepsilon \langle \Gamma_{22}, B \rangle_{\HS} - \tr\left(\frac{\varepsilon}{2}I_d +\left(BM_{\varepsilon}^T AM_{\varepsilon}+ \frac{\varepsilon^2}{4}I_d \right)^{1/2} \right) \\
		& = \tr(B) + \varepsilon \langle \Gamma_{22}, B \rangle_{\HS} - \tr\left(\left(BM_\varepsilon^T AM_{\varepsilon}+ \frac{\varepsilon^2}{4}I_d \right)^{1/2} \right) - \varepsilon \frac{d}{2}.
	\end{align*}
	Then, the identity
	\begin{equation*}
		M^{-1}A^{-1}\left(AM_{\varepsilon}BM_{\varepsilon}^T+ \frac{\varepsilon^2}{4}I_d  \right)^{1/2} AM_{\varepsilon}=\left(BM_{\varepsilon}^T AM_{\varepsilon}+ \frac{\varepsilon^2}{4}I_d  \right)^{1/2}
	\end{equation*}
	ensures that 
	\begin{equation*}
		\tr\left(\left(BM_{\varepsilon}^T AM_{\varepsilon}+ \frac{\varepsilon^2}{4}I_d \right)^{1/2} \right) = \tr\left(\left(AM_{\varepsilon}BM_{\varepsilon}^T+ \frac{\varepsilon^2}{4}I_d \right)^{1/2} \right).
	\end{equation*}
	We thus have that 
	\begin{align*}
		\langle I_d+\varepsilon \Gamma_{11} - F_{\varepsilon}^\star, A \rangle_{\HS} + \langle I_d + \varepsilon \Gamma_{22} - G_\varepsilon^{\star}, B \rangle_{\HS} & = \tr(A) + \tr(B) +\varepsilon \langle \Gamma_{11}, A \rangle_{\HS}  +\varepsilon \langle \Gamma_{22}, B \rangle_{\HS} \\
		&\quad  - 2 \tr\left( \left(AM_{\varepsilon}BM_{\varepsilon}^T+ \frac{\varepsilon^2}{4}I_d \right)^{1/2} \right) - \varepsilon d.
	\end{align*}

	Regarding the determinant term, we write it as 
	\begin{align*}
		\det \begin{pmatrix}
			F_\varepsilon^\star   & -M_{\varepsilon}    \\
			-M_{\varepsilon}^T&    G_\varepsilon^\star
		\end{pmatrix}& =  \det \begin{pmatrix}
			F_\varepsilon^\star  & -M_{\varepsilon}    \\
			-M_{\varepsilon}^T&    \varepsilon B^{-1}+ M_\varepsilon^{T}(F_\varepsilon^{\star})^{-1}M_\varepsilon
		\end{pmatrix} \\
		& = \det ( F_\varepsilon^{\star}) \det(\varepsilon B^{-1}+ M_\varepsilon^{T}(F_\varepsilon^{\star})^{-1}M_{\varepsilon}- M_\varepsilon^{T}(F_\varepsilon^{\star})^{-1}M) \\
		& = \varepsilon^d\det(A)^{-1}\det(B)^{-1} \det\left( \frac{\varepsilon}{2}I_d + \left(AM_{\varepsilon}BM_{\varepsilon}^T+ \frac{\varepsilon^2}{4}I_d \right)^{1/2} \right).
	\end{align*}
	Taking the logarithm of the determinant yields
	\begin{equation*}
		\varepsilon\log \det\begin{pmatrix}
			F_\varepsilon^\star   & -M_{\varepsilon}    \\
			-M_{\varepsilon}^T&    G_\varepsilon^\star
		\end{pmatrix} =  \varepsilon d \log(\varepsilon) - \varepsilon\log \det(AB) + \varepsilon \log \det\left( \frac{\varepsilon}{2}I_d + \left(AM_{\varepsilon}BM_{\varepsilon}^T+ \frac{\varepsilon^2}{4}I_d \right)^{1/2} \right).
	\end{equation*}
	
	Recalling the additive constant $-\varepsilon (2d \log(\varepsilon)+\log \det(\Sigma^{-1})$, and collecting the pieces we derive
	\begin{align*}
		W_{\Sigma}^{\varepsilon}(\mu, \nu)& = \tr(A) + \tr(B) - 2 \tr\left( \left(AM_{\varepsilon}BM_{\varepsilon}^T+ \frac{\varepsilon^2}{4}I_d \right)^{1/2} \right) + \varepsilon \log \det\left( \frac{\varepsilon}{2}I_d + \left(AM_{\varepsilon}BM_{\varepsilon}^T+ \frac{\varepsilon^2}{4}I_d \right)^{1/2} \right) \\
		&+\varepsilon \tr (\Gamma_{11} A )  +\varepsilon \tr (\Gamma_{22} B ) - \varepsilon \log \det(AB) - \varepsilon d - \varepsilon d \log(\varepsilon) - \varepsilon\log \det(\Sigma^{-1}).
	\end{align*}
\end{proof}

\section{Auxiliary results}\label{sec:auxiliary}

\begin{thm}\cite[p.~24, Thm.~1.9]{villani2021topics}
	\label{thm:duality_convex}
	Let $E$ be a separable normed vector space, $E^*$ its topological dual space, and $f$, $g$ two convex functions defined on $E$ with values in $\mathbb{R} \cup \{+\infty\}$. Denoting by $f^*$ and $g^*$ the Legendre-Fenchel transform of and $f$ and $g$ respectively; if there exists a point $z_0 \in E$ such that $f(z_0) < +\infty$, $g(z_0) < +\infty$ and $f$ is continuous at $z_0$, then 
	\begin{equation}
		\inf_{x \in E}\{ f(x) + g(x)\} = \max_{z \in E^*}\{- f^*(-z) - g^*(z)\}. 
	\end{equation}
\end{thm}

On the space of squared matrices $M_d(\mathbb{R})$, the Hilbert-Schmidt (also called Frobenius) scalar product is defined by $\langle A,B \rangle_{\rm HS} := \tr(A B^T)$; and reduces to $\langle A,B \rangle_{\rm HS}  = \tr(AB)$ between symmetric matrices.

\begin{lem}\label{lem:optim_kl_gauss} \cite[p.~34,~Ex.~17]{pardo2018statistical}\label{lem:divergence_covariance} Given a reference centered Gaussian measure $N(\Sigma)$, if $\gamma$ is a centered probability measure with covariance matrix $X$, then
	\begin{equation}
		\KL(N(X)| N(\Sigma)) \leq \KL(\gamma| N(\Sigma)).
	\end{equation}  
\end{lem}

\begin{prop}[Kullback-Leibler divergence]\cite[p.~33, ex.~11]{pardo2018statistical}\label{prop:kl_gaussian}
	For $\mu_0 = N_d(m_0,\Sigma_0)$ and $\mu_1 = N_d(m_1, \Sigma_1)$ two Gaussian measures on $\mathbb{R}^d$, with full-rank covariance matrices $\Sigma_0$ and $\Sigma_1$, the Kullback-Leibler divergence has the closed form expression
	\begin{equation}
		\KL(\mu_1 |\mu_0) = \frac{1}{2} \left[ (m_1 - m_0)^T \Sigma_0^{-1}(m_1 - m_0) + \tr(\Sigma_0^{-1} \Sigma_1 - \Id)   + \log\left(\frac{\det(\Sigma_0)}{\det(\Sigma_1)}\right) \right].
	\end{equation}
We point out that this divergence can be rewritten
\begin{equation}\label{eq:}
	\KL(\mu_1 |\mu_0) = \frac{1}{2} \left[ \langle \Sigma_0^{-1},(m_1-m_0) \otimes (m_1 - m_0) + (\Sigma_1 - \Sigma_0)\rangle _{\rm HS} - \log \det(\Sigma_0^{-1} \Sigma_1)\right], 
\end{equation}
which has a more geometric interpretation.
\end{prop} 

\begin{rem}
In the case where both Gaussian measures $\mu_0$ and $\mu_1$ have zero mean and respective full-rank covariance $\Sigma_0$ and $\Sigma_1$, the Kullback-Leibler divergence simplifies to
\begin{equation}\label{eq:kl_centred_gauss}
	\KL(\mu_1 |\mu_0) = \frac{1}{2} \left[ \langle \Sigma_0^{-1},(\Sigma_1 - \Sigma_0)\rangle _{\rm HS} - \log \det(\Sigma_0^{-1} \Sigma_1)\right].
\end{equation}
\end{rem}

\begin{lem}\label{lem:block_inv}\cite[p.~12, eq.~29]{magnus2019matrix}
    Let $M \in S_{2d}^{++}(\mathbb{R})$ be a positive-definite matrix that we can write by blocks as
    \begin{equation*}
        M =\begin{pmatrix}
            A & C \\
            C^T & B
        \end{pmatrix}, \quad \text{where} \quad A, B \in S_{d}^{++}(\mathbb{R}) \quad \text{and} \quad C \in M_d(\mathbb{R}).
    \end{equation*}
    Then, introducing the Schur complement $S := B - C^T A^{-1} C$, that belongs to $S_d^{++}(\mathbb{R})$, we can write the inverse matrix of $M$ as 
    \begin{equation}\label{eq:block_inv}
        M^{-1} = \begin{pmatrix}
            A^{-1} + A^{-1}CS^{-1}C^TA^{-1} & -A^{-1}C S^{-1} \\
            -S^{-1}C^TA^{-1} & S^{-1}
        \end{pmatrix}.
    \end{equation}
\end{lem}

\begin{thm}\label{thm:positive_mat_positive_schur}\cite[p.~13, Thm.~1.3.3]{bhatia2009positive}
    Let $A, B$ be positive-definite matrices. The block matrix 
\begin{equation*}
    X_C = \begin{pmatrix}
        A & C \\
        C^T & B
    \end{pmatrix}
\end{equation*}
is positive-definite if and only if $B-C^TA^{-1}C$ is positive-definite. 
\end{thm}

\end{document}